\documentclass[11pt]{article}

\usepackage{amsfonts}
\usepackage{amssymb}
\usepackage{amsmath}
\usepackage{mathptmx}
\usepackage{graphicx}
\usepackage{setspace}
\usepackage{amsthm}
\usepackage[labelfont=bf]{caption}
\usepackage[left=2.5cm,top=2.0cm,right=2.5cm,bottom=2.0cm]{geometry}

\theoremstyle{plain}
\newtheorem{thm}{Theorem}[section] 
\newtheorem{defn}[thm]{Definition} 
\newtheorem{lem}[thm]{Lemma} 
\newtheorem{pro}[thm]{Proposition} 
\newtheorem{rem}[thm]{Remark} 
\newtheorem{co}[thm]{Corollary} 

\begin{document}
\begin{center}
\section*{Mixing properties of multivariate infinitely divisible random fields}
\subsection*{Riccardo Passeggeri\footnote[1]{Department of Mathematics, Imperial College London, UK. Email: riccardo.passeggeri14@imperial.ac.uk} and Almut E.~D.~Veraart\footnote[2]{Department of Mathematics, Imperial College London, UK. Email: a.veraart@imperial.ac.uk}}
\end{center}
\begin{abstract}
In this work we present different results concerning mixing properties of multivariate infinitely divisible (ID) stationary random fields. First, we derive some necessary and sufficient conditions for mixing of stationary ID multivariate random fields in terms of their spectral representation. Second, we prove that (linear combinations of independent) mixed moving average fields are mixing. Further, using a simple modification of the proofs of our results we are able to obtain weak mixing versions of our results. Finally, we prove the equivalence of ergodicity and weak mixing for multivariate ID stationary random fields.
\\
\\
\textbf{Key words:} multivariate random field, mixing, weak mixing, ergodicity, infinitely divisible, mixed moving average.
\end{abstract}
\section{Introduction}
In 1970 in his fundamental work \cite{Ma}, Maruyama provided pivotal results for infinitely divisible (ID) processes. Among them, he proved that under certain conditions, known afterwards as Maruyama conditions, these processes are mixing (see Theorem 6 of \cite{Ma}). After him various authors contributed on this line of research, see for example Gross \cite{Gross} and Kososzka and Taqqu \cite{KoTa}. In 1996 Rosinski and Zak extended Maruyama results proving that the a stationary ID process $(X_{t})_{t\in\mathbb{R}}$ is mixing if and only if $\lim\limits_{t\rightarrow\infty}\mathbb{E}\left[e^{i(X_{t}-X_{0})} \right] =\mathbb{E}\left[e^{iX_{0}} \right]\mathbb{E}\left[e^{-iX_{0}} \right]$, provided the L\'{e}vy measure of $X_{0}$ has no atoms in $2\pi\mathbb{Z}$. More recently, Fuchs and Stelzer \cite{FuSt} extended some of the main results of Rosinski and Zak to the multivariate case. Parallel to this line of research, new developments have been obtained for ergodic and weak mixing properties of infinitely divisible \textit{random fields}. In particular, see Roy \cite{Roy2007} and \cite{Roy2009} for Poissonian ID random fields and Roy \cite{Roy2010}, Roy and Samorodnitsky \cite{RoySam} and \cite{WRS} for $\alpha$-stable univariate random fields.

In the present work we fill an important gap by extending the results of Maruyama \cite{Ma}, Rosinski and Zak \cite{RoZa}, and Fuchs \& Stelzer \cite{FuSt} to the multivariate random field case. First, this is crucial for applications since many of them consider a multidimensional domain composed by both spatial and temporal components (and not just temporal ones). This is typically the case for many physical systems, like turbulences (e.g.~\cite{BarSch} \cite{BarBenVar}), and in econometrics (see the models based on panel data). Second, with the present work we also close the gap between the two lines of research presented above by focusing on the more general case of multivariate stationary ID random fields.

On the modelling/application level, we prove that multivariate mixed moving average fields are mixing. This is a relevant result since L\'{e}vy driven moving average fields are extensively used in many applications throughout different disciplines, like brain imaging \cite{J}, tumour growth \cite{BarSch2} and turbulences \cite{BarSch},\cite{BarSch2}, among many.

Moreover, we discuss conditions which ensure that a multivariate random fields is weakly mixing. First, we show that the results obtained can be modified to obtain similar results for the weak mixing case. Second, we prove that a multivariate stationary ID random field is weak mixing if and only if it is ergodic.

The present work is structured as follows. In Section 2, we discuss some preliminaries on mixing and derive the mixing conditions for multivariate ID stationary random fields. In addition, we study some extensions and other related results. In Section 3 we prove that (sums of independent) mixed moving averages (MMA) are mixing, including MMA with an extended subordinated basis. In Section 4 we obtain weak mixing versions of the results obtained in Section 2 and we prove the equivalence between ergodicity and weak mixing for stationary ID random fields.
\\ In order to simplify the exposition, we decided to put long proofs in the appendices. 
\section{Preliminaries and results on mixing conditions}
In this section we analyse mixing conditions for stationary infinite divisible random fields. We work with the probability space $(\Omega,\mathcal{F},\mathbb{P})$ and the measurable space $(\mathbb{R}^{d},\mathcal{B}(\mathbb{R}^{d}))$, where $\mathcal{B}(\mathbb{R}^{d})$ is the Borel $\sigma$-algebra on the vector field $\mathbb{R}^{d}$. We write $\mathcal{L}(X_{t})$ for the distribution, or law, of the random variable $X_{t}$. Now, let $(\theta_{t})_{t\in\mathbb{R}^{l}}$ be a measure preserving $\mathbb{R}^{l}$ action on $(\Omega,\mathcal{F},\mathbb{P})$. Consider the random field $X_{t}(\omega)=X_{0}\circ\theta_{t}(\omega)$, $t\in\mathbb{R}^{l}$. The random field $(X_{t})_{t\in\mathbb{R}^{l}}$ defined in this way is stationary and, conversely, any stationary measurable random field can be expressed in this form. Further, we have, with a little bit of abuse of notation, $\theta_{v}(B):=\{\theta_{v}(\omega)\in\Omega:\omega\in B\}=\{\omega'\in\Omega:X_{0}(\omega')=X_{v}(\omega) \enspace \textnormal{for} \enspace\omega\in B\}$. \\
Then $(X_{t})_{t\in\mathbb{R}^{l}}$ is mixing if and only if (see Wang, Roy and Stoev \cite{WRS} equation (4.4)):
\begin{equation}
\lim\limits_{n\rightarrow\infty}\mathbb{P}(A\cap\theta_{t_{n}}(B))=\mathbb{P}(A)\mathbb{P}(B),
\end{equation}
for all $A,B\in\sigma_{X}$ and all $(t_{n})_{n\in\mathbb{N}}\in\mathcal{T}$, where $\sigma_{X}:=\sigma(\{X_{t}:t\in\mathbb{R}^{l}\})$ is the $\sigma$-algebra generated by $(X_{t})_{t\in\mathbb{R}^{l}}$ and $\mathcal{T}:=\left\{(t_{n})_{n\in\mathbb{N}}\subset\mathbb{R}^{l}:\lim\limits_{n\rightarrow\infty}\|t_{n}\|_{\infty}=\infty \right\}$.\\  The following definition is based on the characteristic function of $(X_{t})_{t\in\mathbb{R}^{l}}$ (see \cite{WRS} equation (A.6)):
\begin{equation}\label{mix}
\lim\limits_{n\rightarrow\infty}\mathbb{E}\left[\exp\left(i\sum_{j=1}^{r}\beta_{j}X_{s_{j}}\right)\exp\left(i\sum_{k=1}^{q}\gamma_{k}X_{p_{k}+t_{n}}\right)\right]=\mathbb{E}\left[\exp\left(i\sum_{j=1}^{r}\beta_{j}X_{s_{j}}\right)\right]\mathbb{E}\left[\exp\left(i\sum_{k=1}^{q}\gamma_{k}X_{p_{k}}\right)\right],
\end{equation}
for all $r,q\in\mathbb{N},\beta_{j},\gamma_{k}\in\mathbb{R},p_{j},s_{k}\mathbb{R}^{l}$ and $(t_{n})_{n\in\mathbb{N}}\in\mathcal{T}$.
\noindent Further, for the multivariate (or $\mathbb{R}^{d}$-valued) random field we have the following definition based on the characteristic function of $(X_{t})_{t\in\mathbb{R}^{l}}$.
\begin{defn}\label{definitionMIX}Let $(X_{t})_{t\in\mathbb{R}^{l}}$ be an $\mathbb{R}^{d}$-valued stationary random field. Then $(X_{t})_{t\in\mathbb{R}^{l}}$ is said to be mixing if for all $\lambda=(s_{1},...,s_{m})',\mu=(p_{1},...,p_{m})'\in\mathbb{R}^{ml}$ and $\theta_{1},\theta_{2}\in\mathbb{R}^{md}$
\begin{equation}\label{MIX}
\lim\limits_{n\rightarrow\infty}\mathbb{E}\left[\exp\left(i\langle\theta_{1},X_{\lambda}\rangle+i\langle\theta_{2},X_{\tilde{\mu}}\rangle\right)\right]=\mathbb{E}\left[\exp\left(i\langle\theta_{1},X_{\lambda}\rangle\right)\right]\mathbb{E}\left[\exp\left(i\langle\theta_{2},X_{\mu}\rangle\right)\right]
\end{equation}
where $X_{\lambda}:=(X_{s_{1}}',...,X_{s_{m}}')'\in\mathbb{R}^{md}$ and
$\tilde{\mu}=(p_{1}+t_{n},...,p_{m}+t_{n})'$, where $(t_{n})_{n\in\mathbb{N}}$ is any sequence in $\mathcal{T}$.
\end{defn}
\noindent We are now ready to state our first result.
\begin{thm}\label{THEOREM}
Let $(X_{t})_{t\in\mathbb{R}^{l}}$, with $l\in\mathbb{N}$, be an $\mathbb{R}^{d}$-valued strictly stationary infinite divisible random field such that $Q_{0}$, the L\'{e}vy measure of $\mathcal{L}(X_{0})$, satisfies $Q_{0}(\{x=(x_{1},...,x_{d})'\in\mathbb{R}^{d}:\exists j\in\{1,...,d\},x_{j}\in2\pi\mathbb{Z}\})=0$. Then $(X_{t})_{t\in\mathbb{R}^{l}}$ is mixing if and only if
\begin{equation}\label{MAINeq}
\lim\limits_{n\rightarrow\infty}\mathbb{E}\left[e^{i(X_{t_{n}}^{(j)}-X_{0}^{(k)})}\right]=\mathbb{E}\left[e^{iX_{0}^{(j)}}\right]\cdot\mathbb{E}\left[e^{-iX_{0}^{(k)}}\right],
\end{equation}
for any $j,k=1,..,d$ and for any sequence $(t_{n})_{n\in\mathbb{N}}\in\mathcal{T}$.
\end{thm}
\noindent The above theorem relies on the following result, which is the multivariate random field extension of the Maruyama conditions (see Theorem 6 of \cite{Ma}).
\begin{thm}\label{LAST-THEOREM}
Let $(X_{t})_{t\in\mathbb{R}^{l}}$ be an $\mathbb{R}^{d}$-valued strictly stationary infinite divisible random field. Then $(X_{t})_{t\in\mathbb{R}^{l}}$ is mixing if and only if
\\
\\
(MM1) the covariance matrix function $\Sigma(t_{n})$ of the Gaussian part of $(X_{t_{n}})_{t_{n}\in\mathbb{R}^{l}}$ tends to 0, as $n\rightarrow\infty$,
\\
\\
(MM2') $\lim\limits_{n\rightarrow\infty}Q_{0t_{n}}(\|x\|\cdot\|y\|>\delta)=0$ for any $\delta>0$, where $Q_{0t_{n}}$ is the L\'{e}vy measure of $\mathcal{L}(X_{0},X_{t_{n}})$ on $(\mathbb{R}^{d},\mathcal{B}(\mathbb{R}^{d}))$,
\\
\\
where $(t_{n})_{n\in\mathbb{N}}$ is any sequence in $\mathcal{T}$.
\end{thm}
\noindent Notice that the above conditions are fewer than the Maruyama conditions. This is because we used the following lemma, which is a multivariate random field extension of Lemma 1 of \cite{Magd} and Lemma 2.2 of \cite{FuSt}.
\begin{lem}\label{LEMMA}
Assume that $\lim\limits_{n\rightarrow\infty}Q_{0t_{n}}(\|x\|\cdot\|y\|>\delta)=0$ holds for any $\delta>0$, where $Q_{0t_{n}}$ is the L\'{e}vy measure of $\mathcal{L}(X_{0},X_{t_{n}})$ on $(\mathbb{R}^{d},\mathcal{B}(\mathbb{R}^{d}))$, and $(t_{n})_{n\in\mathbb{N}}\in\mathbb{R}^{l}$. Then one has
\begin{equation*}
\lim\limits_{n\rightarrow\infty}\int_{0<\|x\|^{2}+\|y\|^{2}\leq 1}\|x\|\cdot\|y\|Q_{0t_{n}}(dx,dy)=0.
\end{equation*}
\end{lem}

\subsection{Related results and extensions}
In this section, we present different results which follow from, are related to or extend the theorems presented in the previous section.
\\The first result is a corollary which follows immediately from Theorem \ref{THEOREM}, and states that a multivariate random field is mixing if and only if its components are pairwise mixing.
\begin{co}\label{1of3corollary}
An $\mathbb{R}^{d}$-valued strictly stationary i.d. random field $X=(X_{t})_{t\in\mathbb{R}^{l}}$ with $Q_{0}(\{x=(x_{1},...,x_{d})'\in\mathbb{R}^{d}:\exists j\in\{1,...,d\},x_{j}\in2\pi\mathbb{Z}\})=0$ is mixing if and only if the bivariate random fields $(X^{(j)},X^{(k)})$, $j,k\in\{1,...,d\}$, $j<k$, are all mixing.
\end{co}
\begin{proof}
It follows immediately from Theorem \ref{THEOREM}.
\end{proof}
\noindent The following corollary is a generalization of Corollary 2.5 of \cite{FuSt}.
\begin{co}\label{corol}
Let $(X_{t})_{t\in\mathbb{R}^{l}}$ be an $\mathbb{R}^{d}$-valued strictly stationary ID random field. Then with the previous notation, $(X_{t})_{t\in\mathbb{R}^{l}}$ is mixing if and only if
\begin{equation}\label{corollarycondition}
\lim\limits_{n\rightarrow\infty}\Big\{\|\Sigma(t_{n})\|+\int_{\mathbb{R}^{2d}}\min (1,\|x\|\cdot\|y\|)Q_{0t_{n}}(dx,dy)\Big\}=0,
\end{equation}
for any $(t_{n})_{n\in\mathbb{N}}\in\mathcal{T}$.
\end{co}
\begin{proof}
We can follow the argument by \cite{FuSt}. To this end, note that if we assume that $(\ref{corollarycondition})$ holds, then conditions $(MM1)$ and $(MM2')$ hold and, thus, Theorem \ref{LAST-THEOREM} implies that $(X_{t})_{t\in\mathbb{R}^{l}}$ is mixing.\\
For the other direction assume that $(X_{t})_{t\in\mathbb{R}^{l}}$ is mixing then by Theorem \ref{LAST-THEOREM} condition $(MM1)$ holds. Furthermore, for every $\delta>0$ with $Q_{jk}(\partial K_{\delta})=0$ and any $j,k=1,...,d$, $(cf. \,\,\,(\ref{TRUNCATION}))$,
\begin{equation}\label{TRUNCATIONNN}
Q_{0t_{n}}^{(jk)}\Big|_{K_{\delta}^{c}}\rightharpoonup Q_{jk}\Big|_{K_{\delta}^{c}} \enspace \enspace as \enspace n\rightarrow\infty.
\end{equation}
for any $(t_{n})_{n\in\mathbb{N}}\in\mathcal{T}$, where the symbol ``$\rightharpoonup$" means convergence in the weak topology. In addition, we know that the L\'{e}vy measures $Q_{jk}$ are concentrated on the axes of $\mathbb{R}^{2}$. Now consider a $\delta>0$ such that conditions $(\ref{Epsilon})$ and $(\ref{TRUNCATIONNN})$ hold. Then we have
\begin{equation*}
\lim\sup\limits_{n\rightarrow\infty}\int_{\mathbb{R}^{2}}\min(1,|xy|)Q_{0t_{n}}^{(jk)}(dx,dy)\leq\epsilon+\lim\sup\limits_{n\rightarrow\infty}\int_{B^{c}_{\delta}}\min(1,|xy|)Q_{0t_{n}}^{(jk)}(dx,dy)=\epsilon.
\end{equation*}
Letting $\epsilon\searrow0$ we obtain that $\lim\sup\limits_{n\rightarrow\infty}\int_{\mathbb{R}^{2}}\min(1,|xy|)Q_{0t_{n}}^{(jk)}(dx,dy)=0$ for any $j,k=1,...,d$. Finally,
\begin{equation*}
\int_{\mathbb{R}^{2d}}\min\left(1,\sum_{k=1}^{d}|x_{k}|\cdot\sum_{j=1}^{d}|y_{j}|\right)Q_{0t_{n}}(dx,dy)\leq \sum_{j,k=1}^{d}\int_{\mathbb{R}^{2d}}\min(1,|x_{k}y_{j}|)Q_{0t_{n}}(dx,dy)
\end{equation*}
\begin{equation*}
=\sum_{j,k=1}^{d}\int_{\mathbb{R}^{2}}\min(1,|xy|)Q_{0t_{n}}^{(jk)}(dx,dy)\rightarrow 0, \quad as \quad n\rightarrow0.
\end{equation*}
Therefore, this implies that
\begin{equation*}
\lim\limits_{n\rightarrow\infty}\int_{\mathbb{R}^{2d}}\min (1,\|x\|\cdot\|y\|)Q_{0t_{n}}(dx,dy),
\end{equation*}
for any $(t_{n})_{n\in\mathbb{N}}\in\mathcal{T}$, hence we obtain that condition $(\ref{corollarycondition})$ is satisfied.
\end{proof}
\noindent The next two results are a reformulation of Theorem \ref{THEOREM}. However, the first requires a short preliminary introduction, which will be useful for Section 4 as well. Recall that a \textit{codifference} $\tau(X_{1},X_{2})$ of an ID real bivariate random vector $(X_{1},X_{2})$ is defined as follows
\begin{equation*}
\tau(X_{1},X_{2}):=\log\mathbb{E}\Big[e^{i(X_{1}-X_{2})}\Big]-\log\mathbb{E}\Big[e^{iX_{1}}\Big]-\log\mathbb{E}\Big[e^{-iX_{2}}\Big],
\end{equation*}
where $\log$ is the distinguished logarithm as defined in $\cite{Sato}$ p. 33.
Following \cite{FuSt} we recall that the \textit{autocodifference} function for an $\mathbb{R}^{d}$-valued strictly stationary ID process $(X_{t})_{t\in\mathbb{R}}$ is defined as $\tau(t)=\left(\tau^{(jk)}(t)\right)_{j,k=1,...,d}$ with $\tau^{(jk)}(t):=\tau\left(X_{0}^{(k)},X_{t}^{(j)}\right)$. For an $\mathbb{R}^{d}$-valued strictly stationary ID random field $(X_{t})_{t\in\mathbb{R}^{l}}$ the \textit{autocodifference field} $\tau(t)$ is defined as $\tau(t)=\left(\tau^{(jk)}(t)\right)_{j,k=1,...,d}$ with $\tau^{(jk)}(t):=\tau\left(X_{0}^{(k)},X_{t}^{(j)}\right)$, where $t\in\mathbb{R}^{l}$.
\begin{co}\label{3of3corollary}
Let $(X_{t})_{t\in\mathbb{R}^{l}}$, with $l\in\mathbb{N}$, be an $\mathbb{R}^{d}$-valued strictly stationary infinite divisible random field such that $Q_{0}$, the L\'{e}vy measure of $\mathcal{L}(X_{0})$, satisfies $Q_{0}(\{x=(x_{1},...,x_{d})'\in\mathbb{R}^{d}:\exists j\in\{1,...,d\},x_{j}\in2\pi\mathbb{Z}\})=0$. Then $(X_{t})_{t\in\mathbb{R}^{l}}$ is mixing if and only if $\tau(t_{n})\rightarrow0$ as $n\rightarrow\infty$ for any sequence $(t_{n})_{n\in\mathbb{N}}\in\mathcal{T}$.
\end{co}
\begin{proof}
It follows immediately from Theorem \ref{THEOREM}.
\end{proof}
\begin{co}\label{norm-corollary}
Let $(X_{t})_{t\in\mathbb{R}^{l}}$, with $l\in\mathbb{N}$, be an $\mathbb{R}^{d}$-valued strictly stationary infinite divisible random field such that $Q_{0}$, the L\'{e}vy measure of $\mathcal{L}(X_{0})$, satisfies $Q_{0}(\{x=(x_{1},...,x_{d})'\in\mathbb{R}^{d}:\exists j\in\{1,...,d\},x_{j}\in2\pi\mathbb{Z}\})=0$. Then $(X_{t})_{t\in\mathbb{R}^{l}}$ is mixing if and only if
\begin{equation}\label{extension}
\lim\limits_{\|t\|\rightarrow\infty}\mathbb{E}\left[e^{i(X_{t}^{(j)}-X_{0}^{(k)})}\right]=\mathbb{E}\left[e^{iX_{0}^{(j)}}\right]\cdot\mathbb{E}\left[e^{-iX_{0}^{(k)}}\right],
\end{equation}
for any $j,k=1,..,d$, where $\|\cdot\|$ is any norm on $\mathbb{R}^{l}$ (e.g.~the sup or the Euclidean norm) and $t\in\mathbb{R}^{l}$.
\end{co}
\begin{proof}
``$\Rightarrow$": Assume that $(X_{t})_{t\in\mathbb{R}^{l}}$ is mixing. Then by Theorem $\ref{THEOREM}$ we know that
\begin{equation}\label{extensionn}
\lim\limits_{n\rightarrow\infty}\mathbb{E}\left[e^{i(X_{t_{n}}^{(j)}-X_{0}^{(k)})}\right]=\mathbb{E}\left[e^{iX_{0}^{(j)}}\right]\cdot\mathbb{E}\left[e^{-iX_{0}^{(k)}}\right]
\end{equation}
holds for any $j,k=1,..,d$ and for any sequence $(t_{n})_{n\in\mathbb{N}}\in\mathcal{T}$. Now consider the following simple result.
\\ Let $M_{1}=(A_{1},d_{1})$ and $M_{2}=(A_{2},d_{2})$ be two metric spaces. Let $S\subseteq A_{1}$ be an open set of $M_{1}$. Let $f$ be a mapping defined on $S$. Then $\lim\limits_{x\rightarrow c}f(x)=l$ iff for any sequence $(x_{n})_{n\in\mathbb{N}}$ of points in $S$ such that $\forall n\in\mathbb{N}:$ $x_{n\neq c}$ and $\lim\limits_{n\rightarrow\infty}x_{n}=c$ we have $\lim\limits_{n\rightarrow \infty}f(x_{n})=l$.
\\ From this result and from the fact that we are considering any sequence such that $\lim\limits_{n\rightarrow\infty}\|t_{n}\|_{\infty}=\infty$, we obtain equation $(\ref{extension})$.
\\ ``$\Leftarrow$": Assume that $(\ref{extension})$ holds. Then we have that $(\ref{extensionn})$ holds by the result stated above. Then by Theorem $\ref{THEOREM}$ we obtain that $(X_{t})_{t\in\mathbb{R}^{l}}$ is mixing.
\\ Now consider the set $\mathcal{E}:=\left\{(t_{n})_{n\in\mathbb{N}}\subset\mathbb{R}^{l}:\lim\limits_{n\rightarrow\infty}\|t_{n}\|=\infty, \textnormal{where }\|\cdot\|\textnormal{ is any norm on }\mathbb{R}^{l}  \right\}$. Notice that any sequence $(t_{n})_{n\in\mathbb{N}}\in\mathcal{T}$ belongs to $\mathcal{E}$ and vice versa, because on the finite dimensional vector space $\mathbb{R}^{l}$ any norm $\|\cdot\|_{a}$ is equivalent to any other norm $\|\cdot\|_{b}$. Hence, we obtain that equation $(\ref{extensionn})$ holds for any $(t_{n})_{n\in\mathbb{N}}\in\mathcal{E}$, and by applying the argument above we obtain our result. 
\end{proof}
\begin{rem}
It is possible to see that the extension that we have done in the above corollary, can be applied to all our results that holds for ``any sequence $(t_{n})_{n\in\mathbb{N}}\in\mathcal{T}$".
\end{rem}
\noindent The next result is a multivariate and random field extension of Theorem 2 of Rosinski and Zak \cite{RoZa} and it will help us to generalise Theorem \ref{THEOREM}.
\begin{thm}\label{atom}
Let $(X_{t})_{t\in\mathbb{R}^{l}}$ be an $\mathbb{R}^{d}$-valued stationary ID random field such that $Q_{0}$, the L\'{e}vy measure of $X_{0}$, satisfies $Q_{0}(\{x=(x_{1},...,x_{d})'\in\mathbb{R}^{d}:\exists j\in\{1,...,d\},x_{j}\in2\pi\mathbb{Z}\})\neq 0$. In other words, $Q_{0}$ has atoms in this set. Let
\begin{equation*}
Z=\{z=(z_{1},...,z_{j})\in\mathbb{R}^{d}: z_{j}=2\pi k/y_{j}\enspace\forall j\in\{1,..,d\},\textnormal{ where }k\in\mathbb{Z}\textnormal{ and $y=(y_{1},...,y_{j})$ is an atom of $Q_{0}$}\}.
\end{equation*}
Then $(X_{t})_{t\in\mathbb{R}^{l}}$ is mixing if and only if for some $a=(a_{1},...,a_{d})\in\mathbb{R}^{d}\setminus Z$, with $a_{p}\neq 0$ for $p=1,..,d$,
\begin{equation*}
\lim\limits_{n\rightarrow\infty}\mathbb{E}\left[e^{i(a_{j}X_{t_{n}}^{(j)}-a_{k}X_{0}^{(k)})}\right]=\mathbb{E}\left[e^{ia_{j}X_{0}^{(j)}}\right]\cdot\mathbb{E}\left[e^{-ia_{k}X_{0}^{(k)}}\right],
\end{equation*}
for any $j,k=1,..,d$ and for any sequence $(t_{n})_{n\in\mathbb{N}}\in\mathcal{T}$.
\end{thm}
\begin{proof}Consider an element $a\in\mathbb{R}^{d}\setminus Z$ with $a_{p}\neq 0$ for $p=1,..,d$. We know that the set of atoms of any $\sigma$-finite measure is  a countable set (the proof is straightforward) and that any L\'{e}vy measure is $\sigma$-finite. Hence, the set of atoms of $Q_{0}$ is countable, which implies that $Z$ is countable. This implies that our $a$ exists. Now, let 
\begin{equation*}M_{a}:=
\begin{pmatrix}
  a_{1} & 0 &\cdot\cdot\cdot& 0  \\
  0 & a_{2} &\cdot\cdot\cdot& 0\\
  \vdots & \vdots & \ddots&\vdots\\
  0 & 0&\dots& a_{d}
 \end{pmatrix}.
\end{equation*}
Notice that $M_{a}$ is an invertible $d\times d$ matrix and $X_{t}$ is a $d$-dimensional column vector. We have that $(X_{t})_{t\in\mathbb{R}^{l}}$ is mixing if and only if $(M_{a}X_{t})_{t\in\mathbb{R}^{l}}$ is mixing. This is because by looking at the Definition \ref{definitionMIX} it is enough to show that for every $m\in\mathbb{N},$ $\lambda=(s_{1},...,s_{m})'\in\mathbb{R}^{ml}$ and $\theta=(\theta_{1},...,\theta_{m})'\in\mathbb{R}^{md}$ we have $\langle \theta,M_{a}\star X_{\lambda}\rangle=\langle \tilde{\theta},X_{\lambda}\rangle$, where $M_{a}\star X_{\lambda}:=(M_{a}X_{s_{1}},...,M_{a}X_{s_{m}})'$ and $\tilde{\theta}\in\mathbb{R}^{md}$. Notice that for $m=1$ we have $M_{a}\star X_{t}:=M_{a} X_{t}$, $t\in\mathbb{R}^{l}$. Indeed, we have $\langle \theta,M_{a}\star X_{\lambda}\rangle=\sum_{j=1}^{d}\sum_{k=1}^{m}a_{j}X_{s_{k}}^{(j)}\theta_{jk}=\langle M_{a}\star\theta, X_{\lambda}\rangle =\langle \tilde{\theta},X_{\lambda}\rangle$, where $ M_{a}\star\theta:=(M_{a}\theta_{1},...,M_{a}\theta_{m})'\in\mathbb{R}^{md}$.\\
Now, the L\'{e}vy measure $Q^{a}_{0}$ of $M_{a}X_{0}$ is given by $Q^{a}_{0}(\cdot)=Q_{0}(M_{a}^{-1}(\cdot))$ (see Proposition 11.10 of \cite{Sato}). Since $a\notin Z$, $Q^{a}_{0}$ has no atoms in the set $\{x=(x_{1},...,x_{d})'\in\mathbb{R}^{d}:\exists j\in\{1,...,d\},x_{j}\in2\pi\mathbb{Z}\}$. This is because
\begin{equation*}
Q^{a}_{0}(\{x=(x_{1},...,x_{d})'\in\mathbb{R}^{d}:\exists j\in\{1,...,d\},x_{j}\in2\pi\mathbb{Z}\})
\end{equation*}
\begin{equation*}
=Q_{0}(\{x=(x_{1},...,x_{d})'\in\mathbb{R}^{d}:\exists j\in\{1,...,d\},x_{j}\in 2\pi\mathbb{Z}/a_{j}\}),
\end{equation*}
since $a\notin Z$ then $\exists j\in\{1,...,d\}:a_{j}\neq 2\pi k/y_{j}$ for any $k\in\mathbb{Z}$ and any atom $y$ of $Q_{0}$, hence
\begin{equation*}
=Q_{0}(\{x=(x_{1},...,x_{d})'\in\mathbb{R}^{d}:\exists j\in\{1,...,d\}\textnormal{ such that }x_{j}\neq y_{j},\textnormal{ where y is any atom of $Q_{0}$}\})=0.
\end{equation*}
The equality to zero comes from the fact that the set considered has no intersection with the set of atoms of the measure $Q_{0}$. Finally, by using Theorem \ref{THEOREM} the proof is complete.
\end{proof}
\noindent From this result we have the following generalisation of Theorem \ref{THEOREM}.
\begin{co}\label{law-corollary}
Let $(X_{t})_{t\in\mathbb{R}^{l}}$ be an $\mathbb{R}^{q}$-valued stationary ID random field. Then $(X_{t})_{t\in\mathbb{R}^{l}}$ is mixing if and only if $\mathcal{L}(X_{t_{n}}-X_{0})\stackrel{n\rightarrow\infty}{\rightarrow}\mathcal{L}(X_{0}-X_{0}')$ for any $j,k=1,..,d$ and for any sequence $(t_{n})_{n\in\mathbb{N}}\in\mathcal{T}$, where $X_{0}'$ is an independent copy of $X_{0}$.
\end{co}
\begin{proof}
This result is an immediate consequence of Theorem \ref{THEOREM}, Corollary \ref{1of3corollary} and Theorem \ref{atom}.
\end{proof}
\noindent We end this section with a simple general result which will also be useful for the next section.
\begin{pro}\label{proMMA} Let $(X_{t})_{t\in\mathbb{R}^{l}}$ be a linear combination of independent, stationary, ID and mixing random fields. In other words, let $r\in\mathbb{N}$ and let $(X_{t})_{t\in\mathbb{R}^{l}}\stackrel{d}{=}(\sum_{k=1}^{r}Y^{k}_{t})_{t\in\mathbb{R}^{l}}$, where $( Y^{k}_{t})_{t\in\mathbb{R}^{l}}$, $k=1,...,r$, are independent $\mathbb{R}^{q}$-valued stationary, ID and mixing random fields. Then $(X_{t})_{t\in\mathbb{R}^{l}}$ is stationary, ID and mixing.
\end{pro}
\section{Mixed moving average field}
In this section we will focus on a specific random field: the mixed moving average (MMA) random field. Before introducing this random field we need to recall the definition of an $\mathbb{R}^{d}$-valued L\'{e}vy basis and the related integration theory. L\'{e}vy basis are also called \textit{infinitely divisible independently scattered random measures} in the literature. In the following let $S$ be a non-empty topological space, $\mathcal{B}(S)$ be the Borel-$\sigma$-field on $S$ and $\pi$ be some probability measure on $(S,\mathcal{B}(S))$. We denote by $\mathcal{B}_{0}(S\times\mathbb{R}^{l})$ the collection of all Borel sets in $S\times\mathbb{R}^{l}$ with finite $\pi\otimes\lambda^{l}$-measure, where $\lambda^{l}$ denotes the $l$-dimensional Lebesgue measure.
\begin{defn}
A d-dimensional L\'{e}vy basis on $S\times\mathbb{R}^{l}$ is an $\mathbb{R}^{d}$-valued random measure $\Lambda=\{\Lambda(B):B\in\mathcal{B}_{0}(S\times\mathbb{R}^{l})\}$ satisfying:
\\(i) the distribution of $\Lambda(B)$ is infinitely divisible for all $B\in\mathcal{B}_{0}(S\times\mathbb{R}^{l})$,
\\(ii) for an arbitrary $n\in\mathbb{N}$ and pairwise disjoint sets $B_{1},...,B_{n}\in\mathcal{B}_{0}(S\times\mathbb{R}^{l})$ the random variables $\Lambda(B_{1}),..,\Lambda(B_{n})$ are independent,
\\(iii) for any pairwise disjoint sets $B_{1},B_{2},...\in\mathcal{B}_{0}(S\times\mathbb{R}^{l})$ with $\bigcup_{n\in\mathbb{N}}B_{n}\in\mathcal{B}_{0}(S\times\mathbb{R}^{l})$ we have, almost surely, $\Lambda(\bigcup_{n\in\mathbb{N}}B_{n})=\sum_{n\in\mathbb{N}}\Lambda(B_{n})$.
\end{defn}
\noindent Throughout this section we shall restrict ourselves to time-homogeneous and factorisable L\'{e}vy bases, i.e.~L\'{e}vy bases with characteristic function given by
\begin{equation}\label{LevyBasis}
\mathbb{E}\bigg[e^{i\langle \theta,\Lambda(B)\rangle}\bigg]=e^{\psi(\theta)\Pi(B)},
\end{equation}
for all $\theta\in\mathbb{R}^{d}$ and $B\in\mathcal{B}_{0}(S\times\mathbb{R}^{l})$, where $\Pi=\pi\otimes\lambda^{l}$ is the product measure of the probability measure $\pi$ on $S$ and the Lebesgue measure $\lambda^{l}$ on $\mathbb{R}^{l}$ and
\begin{equation*}
\psi(\theta)=i\langle\gamma,\theta\rangle-\frac{1}{2}\langle\theta,\Sigma\theta\rangle+\int_{\mathbb{R}^{d}}\left(e^{i\langle\theta,x\rangle}-1-i\langle\theta,x\rangle\textbf{1}_{[0,1]}(\|x\|) \right)Q(dx)
\end{equation*}
is the cumulant transform of an ID distribution with characteristic triplet $(\gamma,\Sigma,Q)$. We note that the quadruple $(\gamma,\Sigma,Q,\pi)$ determines the distribution of the L\'{e}vy basis completely and therefore it is called the \textit{generating quadruple}. Now, we provide an extension of Theorem 3.2 of \cite{FuSt}, which does not need a proof since it is a combination of Theorem 3.2 of \cite{FuSt} and Theorem 2.7 of \cite{RaRo}. It concerns the existence of integrals with respect to a L\'{e}vy basis.
\begin{rem}
In this section we are considering a $q$-valued random field, since the $d$ is used for the $\mathbb{R}^{d}$-valued L\'{e}vy basis, and we denote by $M_{q\times d}(\mathbb{R})$ the collection of $q\times d$ matrices over the field $\mathbb{R}$.
\end{rem}
\begin{thm}\label{TheoremRajRos}
Let $\Lambda$ be an $\mathbb{R}^{d}$-valued L\'{e}vy basis with characteristic function of the form $(\ref{LevyBasis})$ and let $f:S\times\mathbb{R}^{l}\rightarrow M_{q\times d}(\mathbb{R})$ be a measurable function. Then $f$ is $\Lambda$-integrable as a limit in probability in the sense of Rajput and Rosinski \cite{RaRo}, if and only if
\begin{equation*}
\int_{S}\int_{\mathbb{R}^{l}}\Big\|f(A,s)\gamma+\int_{\mathbb{R}^{d}}f(A,s)x(\normalfont\textbf{1}_{[0,1]}(\|f(a,s)x\|)-\normalfont\textbf{1}_{[0,1]}(\|x\|))Q(dx)\Big\| ds\pi(dA)<\infty,
\end{equation*}
\begin{equation*}
\int_{S}\int_{\mathbb{R}^{l}}\|f(A,s)\Sigma f(A,s)'\|ds\pi(dA)<\infty,\qquad and
\end{equation*}
\begin{equation*}
\int_{S}\int_{\mathbb{R}^{l}}\int_{\mathbb{R}^{d}}\min(1,\|f(A,s)x\|^{2})Q(dx)ds\pi(dA)<\infty.
\end{equation*}
If $f$ is $\Lambda$-integrable, the distribution of $\int_{S}\int_{\mathbb{R}^{l}}f(A,s)\Lambda(dA,ds)$ is infinitely divisible with characteristic triplet $(\gamma_{int},\Sigma_{int},v_{int})$ given by
\begin{equation*}
\gamma_{int}=\int_{S}\int_{\mathbb{R}^{l}}f(A,s)\gamma+\int_{\mathbb{R}^{d}}f(A,s)x(\normalfont\textbf{1}_{[0,1]}(\|f(a,s)x\|)-\normalfont\textbf{1}_{[0,1]}(\|x\|))v(dx) ds\pi(dA),
\end{equation*}
\begin{equation*}
\Sigma_{int}= \int_{S}\int_{\mathbb{R}^{l}}f(A,s)\Sigma f(A,s)'ds\pi(dA),\qquad and
\end{equation*}
\begin{equation*}
v_{int}(B)=\int_{S}\int_{\mathbb{R}^{l}}\int_{\mathbb{R}^{d}}\normalfont\textbf{1}_{B}(f(A,s)x)Q(dx)ds\pi(dA)
\end{equation*}
for all Borel sets $B\subseteq\mathbb{R}^{q}\setminus\{0\}.$
\end{thm}
\begin{proof}
This theorem is a specific representation of Theorem 3.2 of \cite{FuSt} and Theorem 2.7 of \cite{RaRo}. 
\end{proof}
\noindent Let us now introduce the main object of interest of this section: the mixed moving average random field.
\begin{defn}
(Mixed Moving Average Random Field). Let $\Lambda$ be an $\mathbb{R}^{d}$-valued L\'{e}vy basis on $S\times\mathbb{R}^{l}$ and let $f:S\times\mathbb{R}^{l}\rightarrow M_{q\times d}(\mathbb{R})$ be a measurable function. If the random field
\begin{equation*}
X_{t}:= \int_{S}\int_{\mathbb{R}^{l}}f(A,t-s)\Lambda(dA,ds)
\end{equation*}
exists  in the sense of Theorem \ref{TheoremRajRos} for all $t\in\mathbb{R}^{l}$, it is called an $n$-dimensional mixed moving average random field (MMA random field for short). The function f is said to be its kernel function.
\end{defn}
\noindent MMA random field have been discussed in Surgailis \textit{et al.}~\cite{Surg} and Veraart \cite{Ver}. Note that an MMA random field is an ID and strictly stationary random field.
\\
The following lemma is a direct application of Corollary \ref{corol} to our MMA random field case.
\begin{lem}\label{lemmaMMA}
Let $(X_{t})_{t\in\mathbb{R}^{l}}\stackrel{d}{=}(\int_{S}\int_{\mathbb{R}^{l}}f(A,t-s)\Lambda(dA,ds))_{t\in\mathbb{R}^{l}}$ be an MMA random field where $\Lambda$ is an $\mathbb{R}^{d}$-valued L\'{e}vy basis on $S\times\mathbb{R}^{l}$ with generating quadruple $(\gamma,\Sigma, Q,\pi)$ and $f:S\times\mathbb{R}^{l}\rightarrow M_{q\times d}(\mathbb{R})$ is a measurable function. Then $(X_{t})_{t\in\mathbb{R}^{l}}$ is mixing if and only if
\begin{equation*}
\lim\limits_{n\rightarrow\infty}\bigg\{\bigg\|\int_{S}\int_{\mathbb{R}^{l}}f(A,-s)\Sigma f(A,t_{n}-s)'ds\pi(dA)\bigg\|
\end{equation*}
\begin{equation*}
+ \int_{S}\int_{\mathbb{R}^{l}}\int_{\mathbb{R}^{d}}\min(1,\|f(A,-s)x\|\cdot\|f(A,t_{n}-s)x\|)Q(dx)ds\pi(dA)\bigg\}=0,
\end{equation*}
for any $(t_{n})_{n\in\mathbb{N}}\in\mathcal{T}$.
\end{lem}
\noindent The following theorem is the main result of this section, while the next proposition is an extension of it.
\begin{thm}\label{MMA}
Let $(X_{t})_{t\in\mathbb{R}^{l}}\stackrel{d}{=}(\int_{S}\int_{\mathbb{R}^{l}}f(A,t-s)\Lambda(dA,ds))_{t\in\mathbb{R}^{l}}$ be an MMA random field where $\Lambda$ is an $\mathbb{R}^{d}$ valued L\'{e}vy basis on $S\times\mathbb{R}^{l}$ with generating quadruple $(\gamma,\Sigma, Q,\pi)$ and $f:S\times\mathbb{R}^{l}\rightarrow M_{q\times d}(\mathbb{R})$ is a measurable function. Then $(X_{t})_{t\in\mathbb{R}^{l}}$ is mixing.
\end{thm}
\noindent From the above result and Proposition \ref{proMMA}, we have this corollary.
\begin{co}
Sums of independent MMA random fields are stationary, ID and mixing random fields.
\end{co}
\begin{proof}
This corollary is an immediate consequence of Theorem \ref{MMA} and Proposition \ref{proMMA}.
\end{proof}
\begin{rem}
The above corollary holds for any linear combination of independent MMA, including MMA with different L\'{e}vy basis and different parameter space $S$.
\end{rem}
\subsection{Meta-times and subordination}
In this section we give a brief introduction of the concepts of meta-times and subordination, and present a result which is a corollary of Theorem $\ref{MMA}$.
\\
First, we recall the Definition of an homogeneous L\'{e}vy sheet (see \cite{BarPed} definition 2.1). Let $\triangle^{b}_{a}F$ indicate the increments of a function $F$ over an interval $(a,b]\subset\mathbb{R}^{k}_{+}$ and let $a\leq b$ indicate $a^{i}\leq b^{i}$ for $i=1,...,k$ (see \cite{BarPed}), where $\mathbb{R}^{m}_{+}=\{x\in\mathbb{R}^{m}:x^{i}\geq0, i=1,...,m \}$ and $m\in\mathbb{N}$. Let $k,l\in\mathbb{N}$.
\begin{defn}
Let $X=\{X_{t}:t\in\mathbb{R}^{k}_{+} \}$ be a family of random vectors in $\mathbb{R}^{d}$. We say that $X$ is an homogeneous L\'{e}vy sheet if $X_{t}=0$ for all $t\in\{t\in\mathbb{R}^{k}_{+}:t^{j}=0,j=1,...,k \}$ a.s., $\triangle_{a_{1}}^{b_{1}}X,...,\triangle_{a_{n}}^{b_{n}}X$ for $n\geq2$ and $(a_{1},b_{1}],...,(a_{n},b_{n}]\subset\mathbb{R}^{k}_{+}$ are disjoint, $X$ is continuous in probability, $\triangle_{a+t}^{b+t}X\stackrel{d}{=}\triangle_{a}^{b}X$ for all $a,b,t\in\mathbb{R}^{k}$ with $a\leq b$, and $X$ is \normalfont lamp.
\end{defn}
\noindent The concept of \textit{lamp} (i.e.~\textit{limits along monotone paths}) is the analogue of \textit{c\`{a}dl\`{a}g}, but in the multiparameter setting. If $X=\{X_{t}:t\in\mathbb{R}^{k}_{+} \}$ is a homogeneous L\'{e}vy sheet then $\mathcal{L}(\triangle_{a}^{b}X)\in ID(\mathbb{R}^{d})$. 
\\ Now let $X=\{X_{t}:t\in\mathbb{R}^{k}_{+} \}$ be an $\mathbb{R}^{d}$-valued homogeneous L\'{e}vy sheet on $\mathbb{R}^{k}_{+}$ and $\Lambda_{X}=\{\Lambda_{X}(A):A\in\mathcal{B}(\mathbb{R}^{k}_{+}) \}$ be the homogeneous L\'{e}vy basis induced by $X$, namely $\Lambda_{X}([0,t])=X_{t}$ a.s.~for all $t\in\mathbb{R}^{k}_{+}$. Let $T=\{T_{t}:t\in\mathbb{R}^{k}_{+} \}$ be an $\mathbb{R}_{+}$-valued homogeneous L\'{e}vy sheet and $\Lambda_{T}=\{\Lambda_{T}(A):A\in\mathcal{B}(\mathbb{R}^{k}_{+}) \}$ be the nonnegative homogeneous L\'{e}vy basis induced by $T$. Define $\mathcal{F}^{T}=\sigma(\Lambda_{T}(A):A\in\mathcal{B}_{b}(\mathbb{R}^{k}_{+}))$ to be the $\sigma$-field generated by $\Lambda_{T}$. Then there exists a $(\mathcal{F}^{T},\mathcal{B}(\mathbb{R}^{k}_{+}),\mathcal{B}(\mathbb{R}^{k}))$-measurable mapping $\phi_{T}:\Omega\times\mathbb{R}^{k}_{+}\rightarrow\mathbb{R}^{k}$ such that for all $\omega\in\Omega$ and $A\in\mathcal{B}_{b}(\mathbb{R}^{k}_{+})$, the set $\textbf{T}(A)(\omega),$ given by $\textbf{T}(A)(\omega)=\{x\in\mathbb{R}^{k}_{+}:\phi_{T}(\omega,x)\in A \}$ is bounded and
\begin{equation*}
\Lambda_{T}(A)(\omega)=Leb(\textbf{T}(A)(\omega)).
\end{equation*}
For each $\omega$, $\textbf{T}(\cdot)(\omega)$ is called a \textit{meta-time associated with }$\Lambda_{T}(\cdot)(\omega)$. Let $M=\{M(A):A\in\mathcal{B}_{b}(\mathbb{R}^{k}_{+}) \}$ be defined as
\begin{equation*}
M(A)(\omega)=\Lambda_{X}(\textbf{T}(A))(\omega)
\end{equation*}
for all $A\in\mathcal{B}(\mathbb{R}^{k}_{+})$. We say that $M$ appears by \textit{extended subordination of} $\Lambda_{X}$ \textit{by} $\Lambda_{T}$ (or \textit{of $X$ by} $T$). Then by Theorem 5.1 of \cite{BarPed} we have that $M$ ia a homogeneous L\'{e}vy basis.
\\ Therefore, we have the following corollary of Theorem \ref{MMA}.
\begin{co}
Let $X=\{X_{t}:t\in\mathbb{R}^{k} \}$ be an $\mathbb{R}^{d}$-valued homogeneous L\'{e}vy sheet on $\mathbb{R}^{k}$ and $\Lambda_{X}=\{\Lambda_{X}(A):A\in\mathcal{B}(\mathbb{R}^{k}) \}$ be the homogeneous L\'{e}vy basis induced by $X$. Let $T=\{T_{t}:t\in\mathbb{R}^{k} \}$ be an $\mathbb{R}_{+}$-valued homogeneous L\'{e}vy sheet and $\Lambda_{T}=\{\Lambda_{T}(A):A\in\mathcal{B}(\mathbb{R}^{k}) \}$ be the nonnegative homogeneous L\'{e}vy basis induced by $T$. Let $M=\{M(A):A\in\mathcal{B}_{b}(\mathbb{R}^{k}_{+}) \}$ be an extended subordination of  $\Lambda_{X}$ by $\Lambda_{T}$. Let $(Y_{t})_{t\in\mathbb{R}^{l}}\stackrel{d}{=}(\int_{\mathbb{R}^{k-l}}\int_{\mathbb{R}^{l}}f(B,t-s)M(dB,ds))_{t\in\mathbb{R}^{l}}$, where $f:\mathbb{R}^{k-l}\times\mathbb{R}^{l}\rightarrow M_{q\times d}(\mathbb{R})$ is a measurable function. Then $(Y_{t})_{t\in\mathbb{R}^{l}}$ is mixing.
\end{co}
\begin{proof}
It is sufficient to notice that the framework introduced above holds for the case $\mathbb{R}^{k}$ and not just for $\mathbb{R}^{k}_{+}$ (see \cite{BarPed}) and that $M$ is an $\mathbb{R}^{d}$-valued homogeneous L\'{e}vy basis on $\mathbb{R}^{k}$. Then by using Theorem \ref{MMA} we obtain the result.
\end{proof}
\section{Weak mixing and ergodicity}
In this section we will first show how to modify our results to obtain weak mixing version of the results presented before and then prove that for stationary ID random fields ergodicity and weak mixing are equivalent. We start with a definition of a density one set and of weak mixing for stationary random fields.
\begin{defn}
A set $E\subset\mathbb{R}^{l}$ is said to have density zero in $\mathbb{R}^{l}$ with respect to the Lebesgue measure $\lambda$ if
\begin{equation*}
\lim\limits_{T\rightarrow\infty}\frac{1}{(2T)^{l}}\int_{(-T,T]^{l}}\mathbf{1}_{E}(x)\lambda(dx)=0.
\end{equation*}
A set $D\subset\mathbb{R}^{l}$ is said to have density one in $\mathbb{R}^{l}$ if $\mathbb{R}^{l}\setminus D$ has density zero in $\mathbb{R}^{l}$.
\end{defn}
\noindent The class of all sequences on $D$ that converge to infinity will be denoted by
\begin{equation*}
\mathcal{T}_{D}:=\left\{(t_{n})_{n\in\mathbb{N}}\subset\mathbb{R}^{l}\cap D:\lim\limits_{n\rightarrow\infty}\|t_{n}\|_{\infty}=\infty \right\}.
\end{equation*}
\begin{defn}
Consider the random field $X_{t}(\omega)=X_{0}\circ\theta_{t}(\omega)$, $t\in\mathbb{R}^{l}$, where $\{\theta_{t} \}_{t\in\mathbb{R}^{l}}$ is a measure preserving $\mathbb{R}^{l}$-action. Let $\sigma_{X}$ be the $\sigma$-algebra generated by the field $(X_{t})_{t\in\mathbb{R}^{l}}$. We say that $(X_{t})_{t\in\mathbb{R}^{l}}$ is weakly mixing if there exists a density one set $D$ such that
\begin{equation*}
\lim\limits_{n\rightarrow\infty}\mathbb{P}(A\cap\theta_{t_{n}}(B))=\mathbb{P}(A)\mathbb{P}(B),
\end{equation*}
for all $A,B\in\sigma_{X}$ and all $(t_{n})_{n\in\mathbb{N}}\in\mathcal{T}_{D}$.
\end{defn}
\noindent We are now ready to state the weak mixing version of Theorem \ref{THEOREM}.
\begin{thm}
Let $(X_{t})_{t\in\mathbb{R}^{l}}$, with $l\in\mathbb{N}$, be an $\mathbb{R}^{d}$-valued strictly stationary infinite divisible random field such that $Q_{0}$, the L\'{e}vy measure of $\mathcal{L}(X_{0})$, satisfies $Q_{0}(\{x=(x_{1},...,x_{d})'\in\mathbb{R}^{d}:\exists j\in\{1,...,d\},x_{j}\in2\pi\mathbb{Z}\})=0$. Then $(X_{t})_{t\in\mathbb{R}^{l}}$ is mixing if and only if there exists a density one set $D\subset\mathbb{R}^{l}$ such that
\begin{equation*}
\lim\limits_{n\rightarrow\infty}\mathbb{E}\left[e^{i(X_{t_{n}}^{(j)}-X_{0}^{(k)})}\right]=\mathbb{E}\left[e^{iX_{0}^{(j)}}\right]\cdot\mathbb{E}\left[e^{-iX_{0}^{(k)}}\right],
\end{equation*}
for any $j,k=1,..,d$ and for any sequence $(t_{n})_{n\in\mathbb{N}}\in\mathcal{T}_{D}$.
\end{thm}
\begin{proof}
It is possible to see that the argument used in the first part of the proof of Theorem \ref{THEOREM} applies and holds also for the case $(t_{n})_{n\in\mathbb{N}}\in\mathcal{T}_{D}$. Moreover, using Theorem \ref{weakTheorem} the proof is complete.
\end{proof}
\begin{thm}\label{weakTheorem}
Let $(X_{t})_{t\in\mathbb{R}^{l}}$ be an $\mathbb{R}^{d}$-valued strictly stationary infinite divisible random field. Then $(X_{t})_{t\in\mathbb{R}^{l}}$ is mixing if and only if there exists a density one set $D\subset\mathbb{R}^{l}$ such that
\\
\\
(MM1) the covariance matrix function $\Sigma(t_{n})$ of the Gaussian part of $(X_{t_{n}})_{t_{n}\in\mathbb{R}^{l}}$ tends to 0, as $n\rightarrow\infty$,
\\
\\
(MM2') $\lim\limits_{n\rightarrow\infty}Q_{0t_{n}}(\|x\|\cdot\|y\|>\delta)=0$ for any $\delta>0$, where $Q_{0t_{n}}$ is the L\'{e}vy measure of $\mathcal{L}(X_{0},X_{t_{n}})$ on $(\mathbb{R}^{d},\mathcal{B}(\mathbb{R}^{d}))$,
\\
\\
where $(t_{n})_{n\in\mathbb{N}}$ is any sequence in $\mathcal{T}_{D}$.
\end{thm}
\begin{proof}
It is possible to see that the arguments used in the proof of Theorem \ref{LAST-THEOREM} go through for the case $(t_{n})_{n\in\mathbb{N}}\in\mathcal{T}_{D}$.
\end{proof}
\begin{rem}\label{remark} It is possible to obtain weak mixing version also for the following results previously stated: Corollary \ref{1of3corollary}, Corollary \ref{corol}, Corollary \ref{3of3corollary}, Corollary \ref{norm-corollary}, Theorem \ref{atom}, Corollary \ref{law-corollary} and Proposition \ref{proMMA}. The proofs of these results are omitted because they follow exactly the same arguments used in the proofs of their respective results. The only difference is that we have $(t_{n})_{n\in\mathbb{N}}\in\mathcal{T}_{D}$ and not $(t_{n})_{n\in\mathbb{N}}\in\mathcal{T}$ but this does not trigger any change in the arguments used in the proofs of these results.
\end{rem}
\noindent Among these results, we have the following corollary, which is the weak mixing version of Corollary \ref{corol} and it will be useful for our next result: the equivalence between weak mixing and ergodicity for stationary ID random fields.
\begin{co}\label{coroll}
Let $(X_{t})_{t\in\mathbb{R}^{l}}$ be an $\mathbb{R}^{d}$-valued strictly stationary ID random field. Then with the previous notation, $(X_{t})_{t\in\mathbb{R}^{l}}$ is mixing if and only if there exists a density one set $D\subset\mathbb{R}^{l}$ such that
\begin{equation*}
\lim\limits_{n\rightarrow\infty}\Big\{\|\Sigma(t_{n})\|+\int_{\mathbb{R}^{2d}}\min (1,\|x\|\cdot\|y\|)Q_{0t_{n}}(dx,dy)\Big\}=0
\end{equation*}
for any $(t_{n})_{n\in\mathbb{N}}\in\mathcal{T}_{D}$.
\end{co}
\begin{proof}
See discussion in Remark \ref{remark}.
\end{proof}
\noindent In order to prove the equivalence between ergodicity and weak mixing for ID stationary random fields we will need various preliminary results, some of which have already been proven in the literature. We start with two known results.
\begin{lem}\label{lemma-weakmixing}\textnormal{[Lemma 4.3 of the ArXiv last version (i.e.~v3) of \cite{WRS}].}
Let $f:\mathbb{R}^{l}\rightarrow\mathbb{R}$ be nonnegative and bounded. A necessary and sufficient condition for
\begin{equation*}
\lim\limits_{T\rightarrow\infty}\frac{1}{(2T)^{l}}\int_{(-T,T]^{l}}f(t)dt=0
\end{equation*}
is that there exists a subset $D$ of density one in $\mathbb{R}^{l}$ such that
\begin{equation*}
\lim\limits_{n\rightarrow\infty}f(t_{n})=0,\qquad\textnormal{for any}\quad (t_{n})_{n\in\mathbb{N}}\in\mathcal{T}_{D}.
\end{equation*}
\end{lem}
\begin{lem}\label{lemma-mu(0)}\textnormal{[Theorem 2.3.2 of \cite{Cupp}].}
Let $\mu$ be a finite measure on $\mathbb{R}^{l}$. Then
\begin{equation*}
\lim\limits_{T\rightarrow\infty}\frac{1}{(2T)^{l}}\int_{(-T,T]^{l}}\hat{\mu}(t)dt=\mu(\{0\}),
\end{equation*}
where $\hat{\mu}$ denotes the Fourier transform of $\mu$.
\end{lem}
\noindent The following lemma is an adaptation to our framework of Lemma 3 of \cite{1997RoZa}.
\begin{lem}\label{lemma-compact}
Let $(X_{t})_{t\in\mathbb{R}^{l}}$ be an $\mathbb{R}^{d}$-valued stationary ID random field and $Q_{0t}^{jk}$ be the L\'{e}vy measure of $\mathcal{L}(X_{0}^{(j)},X_{t}^{(k)})$. Then for every $\delta>0$ and $j,k=1,...,d$, the family of finite measures of $(Q_{0t}^{jk}|_{K_{\delta}^{c}})_{t\in\mathbb{R}^{l}}$ is weakly relatively compact and
\begin{equation}\label{eq1.4}
\lim\limits_{\delta\rightarrow 0}\sup\limits_{t\in\mathbb{R}^{l}}\int_{K_{\delta}}|xy|Q_{0t}^{jk}(dx,dy)=0,
\end{equation}
where $K_{\delta}=\{(x,y):x^{2}+y^{2}\leq\delta^{2} \}$.
\end{lem}
\begin{proof}
This result comes directly from the proof of Lemma 3 of \cite{1997RoZa}.
\end{proof}
\noindent Now we will investigate the auto-codifference matrix of the $\mathbb{R}^{d}$-valued stationary ID random field $(X_{t})_{t\in\mathbb{R}^{l}}$, which was already introduced in Section 2.1. Consider
\begin{equation*}
\tau(t_{n})=\left(\tau^{(jk)}(t_{n})\right)_{j,k=1,...,d},
\end{equation*}
with
\begin{equation}
\tau^{(jk)}(t_{n}):=\tau\left(X_{0}^{(k)},X_{t_{n}}^{(j)}\right)=\log\mathbb{E}\Big[e^{i(X_{0}^{(k)}-X_{t_{n}}^{(j)})}\Big]-\log\mathbb{E}\Big[e^{iX_{0}^{(k)}}\Big]-\log\mathbb{E}\Big[e^{-iX_{t_{n}}^{(j)}}\Big]
\end{equation}
\begin{equation*}
=\sigma_{t_{n}}^{jk}+\int_{\mathbb{R}^{2}}(e^{ix}-1)\overline{(e^{iy}-1)}Q_{0t_{n}}^{jk}(dx,dy),
\end{equation*}
where $\Gamma^{jk}_{t_{n}}$ is the covariance function of the Gaussian part of $(X_{0}^{(k)},X_{t_{n}}^{(j)})$ and it is given by
\begin{equation*}
\Gamma^{jk}_{t_{n}}= \begin{pmatrix}
   \sigma_{0}^{jk} & \sigma_{t_{n}}^{jk}  \\
   \sigma_{t_{n}}^{jk} & \sigma_{0}^{jk}
  \end{pmatrix}.
 \end{equation*}
\begin{pro}\label{pro-ergodicity}
Let $(X_{t})_{t\in\mathbb{R}^{l}}$ be an $\mathbb{R}^{d}$-valued ID random field (not necessarily stationary). Then for any $j,k=1,...,d$ the function
\begin{equation*}
\mathbb{R}^{l}\times\mathbb{R}^{l}\ni(s,t)\rightarrow\tau^{(jk)}(X^{(j)}_{s},X^{(k)}_{t})\in\mathbb{C}
\end{equation*}
is non-negative definite.
\end{pro}
\begin{proof}
We argue as in Proposition 2 of \cite{1997RoZa}. Without loss of generality let $t\geq s$ with $t,s\in\mathbb{R}^{l}$. As seen above, we have 
\begin{equation}\label{tau-st}
\tau\left(X_{s}^{(k)},X_{t}^{(j)}\right)=\sigma_{t-s}^{jk}+\int_{\mathbb{R}^{2}}(e^{ix}-1)\overline{(e^{iy}-1)}Q_{st}^{jk}(dx,dy).
\end{equation}
Since $(s,t)\rightarrow\sigma_{t-s}^{jk}$ is nonnegative definite because it is a covariance function, it just remains to show that the second element on the RHS of $(\ref{tau-st})$ is non-negative definite. However, this is a consequence of Lemma 4 in \cite{1997RoZa}.
\end{proof}
\noindent We can now state and later prove (see Appendix 3) the second main theorem of this section, which states the equivalence between ergodicity and weak mixing for ID stationary random fields.
\begin{thm}\label{Theorem-ergodicity}
Let $l,d\in\mathbb{N}$. Let $(X_{t})_{t\in\mathbb{R}^{l}}$ be an $\mathbb{R}^{d}$-valued stationary ID random field. Then $(X_{t})_{t\in\mathbb{R}^{l}}$ is ergodic if and only if it is weakly mixing.
\end{thm}
\section{Conclusion}
In this work we derived different results concerning ergodicity and mixing properties of multivariate stationary infinitely divisible random fields. A possible future direction consists of the investigation of statistical properties of the results presented in this paper. For example for multivariate stochastic processes, showing that mixed moving average (MMA) processes are mixing implies that the corresponding moment based estimator (like the generalised method of moments (GMM)) are consistent (see \cite{FuSt}). However, it is not clear that a similar result holds for the random fields case. Other possible directions would be to extend the present results to the case of random fields on manifolds or on infinite dimensional vector spaces. However, the literature is not as developed as for the $\mathbb{R}^{l}$-case and requires further work.
\section*{Appendix A: Proofs of Section 2}
\subsection*{Proof of Theorem \ref{THEOREM}}
This proof is an extension to the random field case of the proof of Theorem 2.1 of \cite{FuSt}.\\
``$\Rightarrow$": Assume $(X_{t})_{t\in\mathbb{R}^{l}}$ to be mixing. This implies
\begin{equation*}
\lim\limits_{n\rightarrow\infty}\mathbb{E}\left[\exp\left(i\langle\theta_{1},X_{0}\rangle+i\langle\theta_{2},X_{t_{n}}\rangle\right)\right]=\mathbb{E}\left[\exp\left(i\langle\theta_{1},X_{0}\rangle\right)\right]\mathbb{E}\left[\exp\left(i\langle\theta_{2},X_{0}\rangle\right)\right],
\end{equation*}
for any $\theta_{1},\theta_{2}\in\mathbb{R}^{d}$. Now, setting $(\theta_{1},\theta_{2})=(-e_{k},e_{j})$, $j,k=1,...,d$ with $e_{j}$ the unit $j$-th vector in $\mathbb{R}^{d}$, equation $(\ref{MAINeq})$ is satisfied.
\\``$\Leftarrow$": Assume that equation $(\ref{MAINeq})$ holds for every $j,k=1,..,d$, then we have
\begin{equation}\label{POSITIVE}
\lim\limits_{n\rightarrow\infty}\mathbb{E}\left[e^{i\left(X_{t_{n}}^{(j)}+X_{0}^{(k)}\right)}\right]=\mathbb{E}\left[e^{iX_{0}^{(j)}}\right]\mathbb{E}\left[e^{iX_{0}^{(k)}}\right],
\end{equation}
for every $j,k=1,..,d$. This is true because of the following reasoning. In particular, we extend the proof of Rosinski and Zak \cite{RoZa}, Theorem 1, to the multivariate random field case. Assume that equation $(\ref{MAINeq})$ holds. We will initially prove the following. For every $Y\in L^{2}(\Omega,\mathcal{F},\mathbb{P})$ (complex valued)
\begin{equation}\label{complex}
\lim\limits_{n\rightarrow\infty}\mathbb{E}\left[e^{iX_{t_{n}}^{(j)}}\bar{Y}\right]=\mathbb{E}\left[e^{iX_{0}^{(j)}}\right]\mathbb{E}\left[\bar{Y}\right],
\end{equation}
for $j=1,..,d$. It is possible to see that equation (\ref{complex}) holds for $Y\in H_{0}:=lin\{1,e^{iX_{t}^{(j)}}, t\in\mathbb{R}^{l}\}:=\{Z\in L^{2}(\Omega,\mathcal{F},\mathbb{P}):Z=a_{0}1+\sum_{i=1}^{n}a_{i}e^{iX_{t_{i}}^{(j)}},t_{i}\in\mathbb{R}^{l},n\in\mathbb{N}\}$. Consider now the $L^{2}$-closure of $H_{0}$ and call it $H$. Then by standard density argument equation $(\ref{complex})$ is true for any $Y\in H$. Now, consider any $Y\in L^{2}(\Omega,\mathcal{F},\mathbb{P})$ (complex valued). We can write $Y=Y_{1}+Y_{2}$, where $Y_{1}\in H$ and $Y_{2}\in H^{\bot}:=\{W\in L^{2}(\Omega,\mathcal{F},\mathbb{P}):\mathbb{E}[Z\bar{W}]=0 \enspace \forall Z\in H \}$, where $\mathbb{E}[\cdot\,\,\bar{\cdot}]$ denotes the inner product (usually written as $<\cdot,\cdot>$) from $L^{2}(\Omega,\mathcal{F},\mathbb{P})\times L^{2}(\Omega,\mathcal{F},\mathbb{P})$ to $\mathbb{R}$. Notice that we are using the conjugate (i.e.~in symbol `` $\bar{\enspace}$ ") because this is how the inner product space over a complex space is defined. Further, notice that we can write $Y=Y_{1}+Y_{2}$ since the $L^{2}$-space endowed with that inner product is a Hilbert space.\\
Since $\mathbb{E}[e^{iX_{t_{n}}}\bar{Y}_{2}]=0$ for every $t_{n}\in\mathbb{R}^{l}$ and $\mathbb{E}[\bar{Y}_{2}]=0$ by definition of $H^{\bot}$, we get that
\begin{equation*}
\lim\limits_{n\rightarrow\infty}\mathbb{E}\left[e^{iX_{t_{n}}}\bar{Y}\right]=\lim\limits_{n\rightarrow\infty}\mathbb{E}\left[e^{iX_{t_{n}}}\bar{Y}_{1}\right]=\mathbb{E}\left[e^{iX_{0}}\right]\mathbb{E}\left[\bar{Y}_{1}\right]=\mathbb{E}\left[e^{iX_{0}}\right]\mathbb{E}\left[\bar{Y}\right].
\end{equation*}
Hence we have equation (\ref{complex}). Putting now $Y=e^{-iX_{0}^{(k)}}$ in (\ref{complex}), for $k=1,..,d$, we obtain
\begin{equation}\label{positive}
\lim\limits_{n\rightarrow\infty}\mathbb{E}\left[e^{i(X_{t_{n}}^{(j)}+X_{0}^{(k)})}\right]=\mathbb{E}\left[e^{iX_{0}^{(j)}}\right]\mathbb{E}\left[e^{iX_{0}^{(k)}}\right],
\end{equation}
which is eq. ($\ref{POSITIVE}$). We now prove that equations $(\ref{MAINeq})$ and $(\ref{POSITIVE})$ imply the multidimensional Maruyama conditions:
\\
\\
(MM1) the covariance matrix function $\Sigma(t_{n})$ of the Gaussian part of $(X_{t_{n}})_{t_{n}\in\mathbb{R}^{l}}$ tends to 0, as $n\rightarrow\infty$, where $(t_{n})_{n\in\mathbb{N}}$ is any sequence in $\mathcal{T}$.
\\
\\
(MM2) $\lim\limits_{n\rightarrow\infty}Q_{0t_{n}}(\|x\|\cdot\|y\|>\delta)=0$ and $\lim\limits_{n\rightarrow\infty}\int_{0<\|x\|^{2}+\|y\|^{2}\leq 1}\|x\|\cdot\|y\|Q_{0t_{n}}(dx,dy)=0$ for any $\delta>0$, where $Q_{0t_{n}}$ is the L\'{e}vy measure of $\mathcal{L}(X_{0},X_{t_{n}})$ on $(\mathbb{R}^{d},\mathcal{B}(\mathbb{R}^{d}))$.
\\
\\
Actually, we will not prove (MM2) but we will prove instead the following condition:
\\
\\(MM2') $\lim\limits_{n\rightarrow\infty}Q_{0t_{n}}(\|x\|\cdot\|y\|>\delta)=0$ for any $\delta>0$, where $Q_{0t_{n}}$ is the L\'{e}vy measure of $\mathcal{L}(X_{0},X_{t_{n}})$ on $(\mathbb{R}^{d},\mathcal{B}(\mathbb{R}^{d}))$.
\\
\\
This is because in Lemma $\ref{LEMMA}$ we will prove that (MM2') implies (MM2).\\
Regarding (MM1), we have the following. Since $(X_{0},X_{t_{n}})$ has a $2d$-dimensional ID distribution, its characteristic function can be written, using the L\'{e}vy-Khintchine formulation for every $\theta=(\theta_{1},\theta_{2})'\in\mathbb{R}^{d}\times\mathbb{R}^{d}$, as (see Theorem 8.1 of \cite{Sato} and the proof of Theorem 2.1 of \cite{FuSt})
\begin{equation*}
\mathbb{E}\left[e^{i\langle\theta_{1},X_{0}\rangle+i\langle\theta_{2},X_{t_{n}}\rangle}\right]=\exp\bigg\{i\left\langle\binom{\theta_{1}}{\theta_{2}},\binom{\alpha^{1}}{\alpha^{2}}\right\rangle-\frac{1}{2}\left\langle\binom{\theta_{1}}{\theta_{2}},\Sigma(t_{n})\binom{\theta_{1}}{\theta_{2}}\right\rangle
\end{equation*}
\begin{equation}\label{KHINTCHINE}
+\int_{\mathbb{R}^{2d}}\left(e^{i<\theta_{1},x>+i<\theta_{2},y>}-1-i<\theta_{1},x>\mathbf{1}_{[0,1]}(\|x\|)-i<\theta_{2},y>\mathbf{1}_{[0,1]}(\|y\|)\right)Q_{0t_{n}}(dx,dy)\bigg\}.
\end{equation}
By substituting $(-e_{k},e_{j})$, $(0,e_{j})$ and $(-e_{k},0)$, $j,k=1,...,d$ for $(\theta_{1},\theta_{2})$ in $(\ref{KHINTCHINE})$ we get the description of equation $(\ref{MAINeq})$ in terms of the covariance matrix function of the Gaussian part of $(X^{(j)}_{t_{n}},X^{(k)}_{0})$ and the L\'{e}vy measure $Q_{0t_{n}}$, namely 
\begin{equation*}
\lim\limits_{n\rightarrow\infty}\mathbb{E}\left[e^{i\left(X_{t_{n}}^{(j)}-X_{0}^{(k)}\right)}\right]\left(\mathbb{E}\left[e^{iX_{0}^{(j)}}\right]\mathbb{E}\left[e^{-iX_{0}^{(k)}}\right]\right)^{-1}
\end{equation*}
\begin{equation*}
=\lim\limits_{n\rightarrow\infty}\exp\bigg\{\sigma_{jk}(t_{n})+\int_{\mathbb{R}^{2d}}\left(e^{i\left(y^{(j)}-x^{(k)}\right)}-e^{iy^{(j)}}-e^{-ix^{(k)}}+1\right)Q_{0t_{n}}(dx,dy)\bigg\}=1
\end{equation*}
for arbitrary $j,k=1,...,d$, where $\sigma_{jk}(t_{n})$ is the $(k,j)$-th element of $\Sigma(t_{n})$. By using the identity Real$(e^{i(-x+y)}-e^{-ix}-e^{iy}+1)=(\cos x-1)(\cos y-1)+\sin x\sin y$ and taking the logarithm on both sides, we get
\begin{equation}\label{COS1}
\lim\limits_{n\rightarrow\infty}\bigg[\sigma_{jk}(t_{n})+\int_{\mathbb{R}^{2d}}\left((\cos x^{(k)}-1)(\cos y^{(j)}-1)+\sin x^{(k)}\sin y^{(j)}\right)Q_{0t_{n}}(dx,dy)\bigg]=0,
\end{equation}
for any $j,k=1,..,d.$ Applying the same argument to $\mathbb{E}\left[e^{i\left(X_{t_{n}}^{(j)}+X_{0}^{(k)}\right)}\right]$ we obtain
\begin{equation}\label{COS2}
\lim\limits_{n\rightarrow\infty}\bigg[-\sigma_{jk}(t_{n})+\int_{\mathbb{R}^{2d}}\left((\cos x^{(k)}-1)(\cos y^{(j)}-1)-\sin x^{(k)}\sin y^{(j)}\right)Q_{0t_{n}}(dx,dy)\bigg]=0,
\end{equation}
for any $j,k=1,..,d$. Putting together equations $(\ref{COS1})$ and $(\ref{COS2})$, due to the consistency of the L\'{e}vy measures (see Proposition 11.10 of \cite{Sato} and the proof of Theorem 2.1 of \cite{FuSt}),
\begin{equation}\label{COS3}
\lim\limits_{n\rightarrow\infty}\int_{\mathbb{R}^{2d}}\left(\cos x^{(k)}-1)(\cos y^{(j)}-1\right)Q_{0t_{n}}(dx,dy)=\lim\limits_{n\rightarrow\infty}\int_{\mathbb{R}^{2}}(\cos x-1)(\cos y-1)Q_{0t_{n}}^{(jk)}(dx,dy)=0,
\end{equation}
for any $j,k=1,..,d$, where $Q_{0t_{n}}^{(jk)}$ denotes the L\'{e}vy measure of $\mathcal{L}(X_{0}^{(k)},X_{t_{n}}^{(j)})$ on $(\mathbb{R}^{2},\mathcal{B}(\mathbb{R}^{2}))$.
\\ Now, for fixed $j,k\in\{1,...,d\}$ the family of measures $\{\mathcal{L}(X_{0}^{(k)},X_{t_{n}}^{(j)})\}_{n\in\mathbb{N}}$ is tight. To see this, note that by letting $K_{r}=\{(x,y)\in\mathbb{R}^{2}:x^{2}+y^{2}\leq r^{2}\}$ and using the stationarity of $(X_{t})_{t\in\mathbb{R}^{l}}$ we have
\begin{equation*}
\mathbb{P}((X_{0},X_{t_{n}})\notin K_{r})\leq\mathbb{P}(|X_{0}|>2^{-\frac{1}{2}}r)+\mathbb{P}(|X_{t_{n}}|>2^{-\frac{1}{2}}r)=2\mathbb{P}(|X_{0}|>2^{-\frac{1}{2}}r),
\end{equation*}
hence $\lim\limits_{r\rightarrow\infty}\sup\limits_{t_{n}\in\mathbb{R}^{l}}\mathbb{P}((X_{0},X_{t_{n}})\notin K_{r})=0$. Since $\{\mathcal{L}(X_{0},X_{t_{n}})\}_{n\in\mathbb{N}}$ is tight, then $\{\mathcal{L}(X_{0}^{(k)},X_{t_{n}}^{(j)})\}_{n\in\mathbb{N}}$ is tight.
Notice that we are proving more than necessary because it is sufficient to prove the above limit for $\sup\limits_{n\in\mathbb{N}}$. Thus, by Prokhorov's theorem we have that the family $\{\mathcal{L}(X_{0}^{(k)},X_{t_{n}}^{(j)})\}_{n\in\mathbb{N}}$ is sequentially compact (in the topology of weak convergence). Choose any sequence $(\tau_{n})_{n\in\mathbb{N}}\in\mathcal{T}$ and let $F_{jk}$ be a cluster (or accumulation) point of the family $\{\mathcal{L}(X_{0}^{(k)},X_{\tau_{n}}^{(j)})\}_{n\in\mathbb{N}}$. Then, using Lemma 7.8 of page 34 of Sato's book \cite{Sato}, we have that $F_{jk}$ is an ID distribution on $\mathbb{R}^{2}$ with some L\'{e}vy measure $Q_{jk}$. Moreover, let $(t_{m})_{m\in\mathbb{N}}$ be a subsequence of $(\tau_{n})_{n\in\mathbb{N}}$ such that
\begin{equation}\label{WEAK-CONVERGENCE}
\mathcal{L}(X_{0}^{(k)},X_{t_{m}}^{(j)})\rightharpoonup F_{jk}, \enspace \enspace as \enspace m\rightarrow\infty.
\end{equation}
Notice that $(t_{m})_{m\in\mathbb{N}}\in\mathcal{T}$ as well. We know that $F_{jk}$ exists by Prokhorov theorem on Euclidean spaces. Then, for every $\delta>0$ with $Q_{jk}(\partial K_{\delta})=0$,
\begin{equation}\label{TRUNCATION}
Q_{0t_{m}}^{(jk)}\Big|_{K_{\delta}^{c}}\rightharpoonup Q_{jk}\Big|_{K_{\delta}^{c}}, \enspace \enspace as \enspace m\rightarrow\infty.
\end{equation}
Since $(\cos x-1)(\cos y-1)\geq0$ and using equations $(\ref{COS3})$ and $(\ref{TRUNCATION})$, we deduce that
\begin{equation*}
0\leq \int_{K^{c}_{\delta}}(\cos x-1)(\cos y-1)Q_{jk}(dx,dy)=\lim\limits_{n\rightarrow\infty}\int_{K^{c}_{\delta}}(\cos x-1)(\cos y-1)Q_{0t_{m}}^{(jk)}(dx,dy)
\end{equation*}
\begin{equation*}
\leq \lim\limits_{n\rightarrow\infty}\int_{\mathbb{R}^{2}}(\cos x-1)(\cos y-1)Q_{0t_{n}}^{(jk)}(dx,dy)=0.
\end{equation*}
Every L\'{e}vy measure $Q_{jk}$ is concentrated on the set of straight lines $\{(x,y)\in\mathbb{R}^{2}:x\in2\pi\mathbb{Z}\enspace or \enspace y\in2\pi\mathbb{Z}\}$. This is because only on these lines the integrand $(\cos x-1)(\cos y-1)$ is zero, otherwise it is positive, and because $\delta$ can be taken arbitrarily small. 
\\ By the stationarity of the process and $(\ref{WEAK-CONVERGENCE})$, the projection of $Q_{jk}$ onto the first and second axis coincides with $Q_{0}^{(k)}$ and $Q_{0}^{(j)}$, respectively, on the complement of every neighbourhood of zero (because $(\ref{TRUNCATION})$ holds on any complement of every neighbourhood of zero). Recall the assumption on $Q_{0}$, i.e. $Q_{0}(\{x=(x_{1},...,x_{d})'\in\mathbb{R}^{d}:\exists j\in\{1,...,d\},x_{j}\in2\pi\mathbb{Z}\})=0$. Hence, we have for every $a\in\mathbb{Z}$, $a\neq0$,
\begin{equation*}
Q_{jk}(\{2\pi a\}\times\mathbb{R})=Q^{(k)}_{0}(\{2\pi a\})=Q_{0}(\mathbb{R}\times...\times\mathbb{R}\times\{2\pi a\}\times\mathbb{R}\times...\times\mathbb{R})
\end{equation*}
\begin{equation*}
\leq Q_{0}(\{x\in\mathbb{R}^{d}:\exists l\in\{1,...,d\},x_{l}\in2\pi\mathbb{Z}\})=0
\end{equation*}
and similarly $Q_{jk}(\mathbb{R}\times \{2\pi a\})=0.$ Therefore, $Q_{jk}$, $j,k=1,...,d$ is concentrated on the axes of $\mathbb{R}^{2}$ and on each of them coincides with $Q_{0}^{(k)}$ and $Q_{0}^{(j)}$. It is important to stress the main ideas of the above argument. First, we showed that $Q_{jk}$ is concentrated only on $\{(x,y):x\in2\pi\mathbb{Z}\enspace \textnormal{or}\enspace y\in2\pi\mathbb{Z}\}$ and then using the assumption of our theorem we showed that only when $k=0$ the measure $Q_{jk}$ is non-zero. Further, the stationarity of the process allows the fact that when we project $Q_{jk}$ on the axes the projections coincide with $Q_{0}^{(k)}$ and $Q_{0}^{(j)}$, avoiding the case of having $Q_{0}^{(k)}$ on one axis and $Q_{t_{m}}^{(j)}$ on the other.\\ Now observe that, for every $t\in\mathbb{R}^{l}$, using the consistency of the L\'{e}vy measure
\begin{equation*}
\int_{K_{\delta}}|xy|Q_{0t}^{(jk)}(dx,dy)\leq \frac{1}{2}\int_{K_{\delta}}(x^{2}+y^{2})Q_{0t}^{(jk)}(dx,dy)\leq\frac{1}{2}\int_{|x|\leq\delta}x^{2}Q_{0t}^{(jk)}(dx,dy)+\frac{1}{2}\int_{|y|\leq\delta}y^{2}Q_{0t}^{(jk)}(dx,dy)
\end{equation*}
\begin{equation}\label{Epsilon}
=\frac{1}{2}\int_{|x|\leq\delta}x^{2}Q_{0}^{(k)}(dx)+\frac{1}{2}\int_{|y|\leq\delta}y^{2}Q_{0}^{(j)}(dy)=\int_{|x|\leq\delta}x^{2}Q_{0}(dx)<\epsilon,
\end{equation}
for any positive $\epsilon$ and any $j,k=1,...,d$, if only $\delta$ is small enough. This implies that, by $(\ref{Epsilon})$ for every $j,k=1,...,d$ we have for all $m\in\mathbb{N}$
\begin{equation*}
\int_{K_{\delta}}|\sin x \sin y|Q_{0t_{m}}^{(jk)}(dx,dy)\leq\int_{K_{\delta}}|xy|Q_{0t_{m}}^{(jk)}(dx,dy)<\epsilon,
\end{equation*}
for sufficiently small $\delta>0$. Since $Q_{jk}$ is concentrated on the axes of $\mathbb{R}^{2}$ then $\int_{K_{\delta}^{c}}\sin x \sin yQ_{jk}(dx,dy)=0$ and using $(\ref{TRUNCATION})$ we get
$\lim\limits_{m\rightarrow\infty}\int_{K_{\delta}^{c}}\sin x \sin yQ_{0t_{m}}^{(jk)}(dx,dy)=0$. Thus,
\begin{equation}\label{Bo}
\lim\limits_{m\rightarrow\infty}\int_{\mathbb{R}^{2}}\sin x \sin yQ_{0t_{m}}^{(jk)}(dx,dy)=0,
\end{equation}
for every $j,k=1,...,d$. Combining equations $(\ref{COS1})$, $(\ref{COS3})$ and $(\ref{Bo})$ we obtain that $\sigma_{jk}(t_{m})\rightarrow0$ as $m\rightarrow\infty$ for all $j,k=1,...,d$. Since $(t_{m})_{m\in\mathbb{N}}$ is a subsequence of an arbitrary sequence $\tau_{n}\in\mathcal{T}$, it follows that $\sigma_{jk}(t_{n})\rightarrow0$ as $n\rightarrow\infty$ and thus $\Sigma(t_{n})\rightarrow0$ as $n\rightarrow\infty$, for any $(t_{n})_{n\in\mathbb{N}}\in\mathcal{T}$. This is because we have used the fact that if a sequence has the property that any subsequence has a further subsequence that converges to the same limit, then the sequence converges to that limit. Hence, condition (MM1) follows. 
\\
\\
To prove (MM2'), observe that, for any $m\in\mathbb{N}$,
\begin{equation*}
Q_{0t_{m}}\left(\|x\|^{2}\cdot\|y\|^{2}>\delta^{2}\right)=Q_{0t_{m}}\left(\sum_{j,k=1}^{d}\left(x^{(k)}y^{(j)}\right)^{2}>\delta^{2}\right)\leq\sum_{j,k=1}^{d}Q_{0t_{m}}^{(jk)}\left(\Big|x^{(k)}y^{(j)}\Big|\geq\frac{\delta}{d}\right).
\end{equation*}
In view of $(\ref{TRUNCATION})$ we also get
\begin{equation*}
\limsup\limits_{m\rightarrow\infty}Q_{0t_{m}}^{(jk)}\left(\Big|x^{(k)}y^{(j)}\Big|\geq\frac{\delta}{d}\right)\leq Q_{jk}\left(\Big|x^{(k)}y^{(j)}\Big|\geq\frac{\delta}{d}\right)=0,
\end{equation*}
for any $\delta>0$ and $j,k=1,...,d$. Hence, $\lim\limits_{m\rightarrow\infty}Q_{0t_{m}}(\|x\|\cdot\|y\|>\delta)=0$ for any $\delta>0$ and together with the fact that $(t_{k})_{k\in\mathbb{N}}$ is a subsequence of an arbitrary sequence $\tau_{n}\in\mathcal{T}$ we obtain condition (MM2').
\\
Now, by Theorem $\ref{LAST-THEOREM}$ the proof is complete.
\subsection*{Proof of Theorem \ref{LAST-THEOREM}}
This proof is an extension to the random field case of the proof of Theorem 2.3 of \cite{FuSt}.\\
``$\Rightarrow$": We have shown in the proof of Theorem $\ref{THEOREM}$ that mixing implies conditions (MM1) and (MM2'). In particular, mixing implies formula $(\ref{MAINeq})$, which implies (MM1) and (MM2').
\\ ``$\Leftarrow$": For the other direction we have the following. Assume that (MM1) and (MM2') hold. Then by Lemma $\ref{LEMMA}$ condition (MM2) holds. We need to prove that the process is mixing, i.e.~for all $m\in\mathbb{N},\lambda=(s_{1},...,s_{m})',\mu=(p_{1},...,p_{m})'\in\mathbb{R}^{ml}$ and $\theta_{1},\theta_{2}\in\mathbb{R}^{md}$
\begin{equation}\label{Mix}
\lim\limits_{n\rightarrow\infty}\mathbb{E}\left[\exp\left(i\langle\theta_{1},X_{\lambda}\rangle+i\langle\theta_{2},X_{\tilde{\mu}}\rangle\right)\right]=\mathbb{E}\left[\exp\left(i\langle\theta_{1},X_{\lambda}\rangle\right)\right]\mathbb{E}\left[\exp\left(i\langle\theta_{2},X_{\mu}\rangle\right)\right],
\end{equation}
where $X_{\lambda}:=(X_{s_{1}}',...,X_{s_{m}}')'\in\mathbb{R}^{md}$ and
$\tilde{\mu}=(p_{1}+t_{n},...,p_{m}+t_{n})'$, where $(t_{n})_{n\in\mathbb{N}}$ is any sequence in $\mathcal{T}$.
\\ The family of $\mathbb{R}^{2md}$-valued distributions of the ID random fields $(X_{\lambda},X_{\tilde{\mu}})$, denoted $\{\mathcal{L}(X_{\lambda},X_{\tilde{\mu}})\}_{n\in\mathbb{N}}$ is tight. Indeed, let $K$ be the $2md$-dimensional ball around the origin with radius of length $\sqrt{2m}a$, then by stationarity of the process $(X_{t})_{t\in\mathbb{R}^{l}}$ we have
\begin{equation*}
\mathbb{P}((X_{\lambda},X_{\tilde{\mu}})\notin K)\leq\sum_{j=1}^{m}\mathbb{P}(\|X_{s_{j}}\|\geq a)+\sum_{j=1}^{m}\mathbb{P}(\|X_{p_{j}+t_{n}}\|\geq a)=2m\sum_{j=1}^{m}\mathbb{P}(\|X_{0}\|\geq a),
\end{equation*}
hence $\lim\limits_{a\rightarrow\infty}\sup\limits_{\lambda,\mu,t_{n}}\mathbb{P}((X_{\lambda},X_{\tilde{\mu}})\notin K)=0$. Again notice that we are proving more than necessary because it is sufficient to prove the above limit for $\sup\limits_{\lambda,\mu,n}$.
Let $(\alpha^{1},\Sigma^{1},Q^{1})$ and $(\alpha^{2},\Sigma^{2},Q^{2})$ be the characteristic triplets of $\mathcal{L}(X_{\lambda})$ and $\mathcal{L}(X_{\mu})$ respectively.
\\
Suppose $F$ with characteristic triplet $(\alpha,R,Q)$ is a cluster point of the distributions of $(X_{\lambda},X_{\tilde{\mu}})$. As in the proof of the previous theorem, we have that $F$ is the limit as $r\rightarrow\infty$ of the distributions of $(X_{\lambda},X_{\tilde{\mu}_{r}})$, where $\tilde{\mu}_{r}=(p_{1}+t_{r},...,p_{m}+t_{r})'$ with $t_{r}$ is a subsequence of $t_{n}$. Let $(\alpha_{r},\Sigma_{r},Q_{r})$ be the characteristic triplets of $(X_{\lambda},X_{\tilde{\mu}_{r}})$. By Lemma 7.8 of \cite{Sato} $F$ is an ID distribution on $\mathbb{R}^{2md}$. We denote by $\Phi_{r}(\theta_{1},\theta_{2})$ the characteristic function of $\mathcal{L}(X_{\lambda},X_{\tilde{\mu}_{r}})$ at the point $(\theta_{1},\theta_{2})\in\mathbb{R}^{md}\times\mathbb{R}^{md}$. The logarithm of $\Phi_{r}(\theta_{1},\theta_{2})$ can be written (see proof of Theorem $\ref{THEOREM}$ and Theorem 2.3.of \cite{FuSt}) as
\begin{equation*}
\log\Phi_{r}(\theta_{1},\theta_{2})=i\left\langle \binom{\theta_{1}}{\theta_{2}},\binom{\alpha^{1}_{r}}{\alpha^{2}_{r}} \right\rangle-\frac{1}{2}\left\langle\binom{\theta_{1}}{\theta_{2}},\Sigma_{r}\binom{\theta_{1}}{\theta_{2}}\right\rangle
\end{equation*}
\begin{equation*}
+\int_{\{\|x\|<\delta,\|y\|<\delta\}}\left(e^{i\langle\theta_{1},x\rangle+i\langle\theta_{2},y\rangle}-1-i\langle\theta_{1},x\rangle\mathbf{1}_{[0,1]}(\|x\|)-i\langle\theta_{2},y\rangle\mathbf{1}_{[0,1]}(\|y\|)\right)Q_{r}(dx,dy)
\end{equation*}
\begin{equation*}
+\int_{\{\|x\|\geq\delta \,\,\, \textnormal{or} \,\,\,\|y\|\geq\delta\}}\left(e^{i\langle\theta_{1},x\rangle+i\langle\theta_{2},y\rangle}-1-i\langle\theta_{1},x\rangle\mathbf{1}_{[0,1]}(\|x\|)-i\langle\theta_{2},y\rangle\mathbf{1}_{[0,1]}(\|y\|)\right)Q_{r}(dx,dy)
\end{equation*}
\begin{equation*}
:=I_{1}+I_{2}+I_{3}+I_{4}.
\end{equation*}
We need to prove that $\log\Phi_{r}(\theta_{1},\theta_{2})\rightarrow\log\Phi_{1}(\theta_{1})+\log\Phi_{2}(\theta_{2})$ as $r\rightarrow\infty$ for all $\theta_{1},\theta_{2}\in\mathbb{R}^{md}$ where $\Phi_{1}$ and $\Phi_{2}$ are the characteristic functions of $X_{\lambda}$ and $X_{\mu}$, respectively.
\\
It is possible to see immediately that $I_{1}=i\langle\alpha^{1},\theta_{1}\rangle+i\langle\alpha^{2},\theta_{2}\rangle$, just by setting $\alpha_{r}=(\alpha^{1}_{r},\alpha^{2}_{r})'=(\alpha^{1},\alpha^{2})'$. Further, by condition (MM1) $I_{2}$ converges to $-\frac{1}{2}\langle\Sigma^{1}\theta_{1},\theta_{1}\rangle-\frac{1}{2}\langle\Sigma^{2}\theta^{2},\theta^{2}\rangle$ as $r\rightarrow\infty$, where $\Sigma^{1}$ and $\Sigma^{2}$ are the $md$-dimensional covariance matrix of $(X_{\lambda})$ and $(X_{\mu})$ respectively.
\\ For $I_{4}$, we have
\begin{equation*}
\lim\limits_{r\rightarrow\infty}I_{4}=\int_{\{\|x\|\geq\delta \,\,\, or \,\,\,\|y\|\geq\delta\}}\left(e^{i\langle\theta_{1},x\rangle+i\langle\theta_{2},y\rangle}-1-i\langle\theta_{1},x\rangle\mathbf{1}_{[0,1]}(\|x\|)-i\langle\theta_{2},y\rangle\mathbf{1}_{[0,1]}(\|y\|)\right)Q(dx,dy)
\end{equation*}
\begin{equation*}
=\int_{\{\|x\|\geq\delta\}}\left(e^{i\langle\theta_{1},x\rangle}-1-i\langle\theta_{1},x\rangle\mathbf{1}_{[0,1]}(\|x\|)\right)Q^{1}(dx)
\end{equation*}
\begin{equation*}
+\int_{\{\|y\|\geq\delta\}}\left(e^{i<\theta_{2},y>}-1-i<\theta_{2},y>\mathbf{1}_{[0,1]}(\|y\|)\right)Q^{2}(dy).
\end{equation*}
This is because of the identity 
\begin{equation*}
\exp\left(\sum_{i=1}^{m}a_{i}+\sum_{i=1}^{m}b_{i}\right)-1=\exp\left(\sum_{i=1}^{m}a_{i}\right)-1+\exp\left(\sum_{i=1}^{m}b_{i}\right)-1,
\end{equation*}
which is valid when $a_{i}b_{j}=0$ for any $1\leq i,j\leq m$, and because, letting $x=\left(x^{(1)'},...,x^{(m)'}\right)\in(\mathbb{R}^{d})^{m}$ and  $y=\left(y^{(1)'},...,y^{(m)'}\right)\in(\mathbb{R}^{d})^{m}$, we have
\begin{equation*}
Q(\|x\|\cdot\|y\|>\delta)\leq \liminf\limits_{r\rightarrow\infty}Q_{k}(\|x\|\cdot\|y\|>\delta)\leq \liminf\limits_{r\rightarrow\infty}\sum_{j,i=1}^{m}Q_{0,t_{r}+p_{i}-s_{j}}\left(|\|x^{(j)}\|\cdot\|y^{(i)}\|>\frac{\delta}{m}\right)\stackrel{(M2)}{=} 0,
\end{equation*} 
for any $\delta>0$, which shows in particular that $Q(\|x\|\cdot\|y\|>0)=0$.
\\
Analogously to $x$ and $y$ we denote by $\theta_{1}^{(j)}$ and $\theta_{2}^{(j)}$ the $j$-th $\mathbb{R}^{d}$-component of $\theta_{1}$ and $\theta_{2}$, respectively. Concerning $I_{3}$, consider the multivariate Taylor expansion of
\begin{equation*}
e^{i\langle\theta_{1},x\rangle+i\langle\theta_{2},y\rangle}-1-i\langle\theta_{1},x\rangle\mathbf{1}_{[0,1]}(\|x\|)-i\langle\theta_{2},y\rangle\mathbf{1}_{[0,1]}(\|y\|)=:g(x,y)
\end{equation*}
with respect to the variable $(x,y)'$ at the point $(x_{0},y_{0})'\equiv0$. For any $\delta>0$ small enough, we obtain
\begin{equation*}
I_{3}=-\frac{1}{2}\Bigg[\int_{\{\|x\|<\delta,\|y\|<\delta\}}\left(\sum_{j=1}^{m}\left\langle\theta_{1}^{(j)},x^{(j)}\right\rangle\right)^{2}+\left(\sum_{j=1}^{m}\left\langle\theta_{2}^{(j)},y^{(j)}\right\rangle\right)^{2}Q_{r}(dx,dy)
\end{equation*}
\begin{equation*}
+2\int_{\{\|x\|<\delta,\|y\|<\delta\}}\left(\sum_{j,k=1}^{m}\left\langle\theta_{1}^{(j)},x^{(j)}\right\rangle\left\langle\theta_{2}^{(k)},y^{(k)}\right\rangle\right)Q_{r}(dx,dy)\Bigg]+R,
\end{equation*}
where $R$ is the reminder and in the integral form it is given by
\begin{equation*}
R=\int_{\{\|x\|<\delta,\|y\|<\delta\}}\frac{1}{6}\int_{0}^{1}(1-t)^{2}D_{x,y}^{3}(g(tx,ty))dtQ_{k}(du),
\end{equation*}
and so
\begin{equation*}
6|R|\leq \int_{\{\|x\|<\delta,\|y\|<\delta\}}\int_{0}^{1}|\langle\theta_{1},tx\rangle+\langle\theta_{2},ty\rangle|^{3}\Big|e^{i\langle\theta_{1},tx\rangle+i\langle\theta_{2},ty\rangle}\Big|dtQ_{r}(dx,dy)
\end{equation*}
\begin{equation*}
= \int_{\{\|x\|<\delta,\|y\|<\delta\}}|\langle\theta_{1},x\rangle+\langle\theta_{2},y\rangle|^{3}Q_{r}(dx,dy)\int_{0}^{1}t^{3}dt=\frac{1}{4}\int_{\{\|x\|<\delta,\|y\|<\delta\}}|\langle\theta_{1},x\rangle+\langle\theta_{2},y\rangle|^{3}Q_{r}(dx,dy)
\end{equation*}
\begin{equation*}
\leq \frac{1}{4}\bigg\|\binom{\theta_{1}}{\theta_{2}}\bigg\|^{3}\int_{\{\|x\|<\delta,\|y\|<\delta\}}\bigg\|\binom{x}{y}\bigg\|^{3}Q_{r}(dx,dy)
\end{equation*}
\begin{equation*}
\leq \frac{1}{4}\bigg\|\binom{\theta_{1}}{\theta_{2}}\bigg\|^{3}2\delta\left(\int_{\{0<\|x\|<\delta\}}\|x\|^{2}Q^{1}(dx)+\int_{\{0<\|y\|<\delta\}}\|y\|^{2}Q^{2}(dx)\right)
\end{equation*}
and thus $6|R|<\epsilon$ for any positive $\epsilon$ if only $\delta$ is sufficiently small. Notice that our estimates are sharper than the ones of \cite{FuSt} because we work with the explicit integral form of the remainder. Moreover, we obtain for every $j,k=1,...,m$ and any $\delta$ small enough that
\begin{equation*}
\int_{\{\|x\|<\delta,\|y\|<\delta\}}\Big|\left\langle\theta_{1}^{(j)},x^{(j)}\right\rangle\left\langle\theta_{2}^{(k)},y^{(k)}\right\rangle\Big|Q_{r}(dx,dy)
\end{equation*}
\begin{equation*}
\leq \Big\|\theta_{1}^{(j)}\Big\|\cdot\Big\|\theta_{2}^{(k)}\Big\|\int_{\{\|x\|<\delta,\|y\|<\delta\}}\Big\|x^{(j)}\Big\|\cdot\Big\|y^{(k)}\Big\|Q_{r}(dx,dy)
\end{equation*}
\begin{equation*}
\leq \Big\|\theta_{1}^{(j)}\Big\|\cdot\Big\|\theta_{2}^{(k)}\Big\|\int_{\{0<\|x^{(j)}\|^{2}+\|y^{(k)}\|^{2}\leq 2\delta^{2}\}}\Big\|x^{(j)}\Big\|\cdot\Big\|y^{(k)}\Big\|Q_{0,t_{r}+p_{k}-s_{j}}\left(dx^{(j)},dy^{(k)}\right)
\end{equation*}
\begin{equation*}
\leq \Big\|\theta_{1}^{(j)}\Big\|\cdot\Big\|\theta_{2}^{(k)}\Big\|\int_{\{0<\|x^{(j)}\|^{2}+\|y^{(k)}\|^{2}\leq 1\}}\Big\|x^{(j)}\Big\|\cdot\Big\|y^{(k)}\Big\|Q_{0,t_{r}+p_{k}-s_{j}}\left(dx^{(j)},dy^{(k)}\right)\stackrel{r\rightarrow\infty}{\rightarrow}0,
\end{equation*}
by using condition (MM2). Finally, we have
\begin{equation*}
\Bigg|\frac{1}{2}\int_{\{\|x\|<\delta,\|y\|<\delta\}}\langle\theta_{1},x\rangle^{2}Q_{r}(dx,dy)+\int_{\{0<\|x\|<\delta\}}e^{i\langle\theta_{1},x\rangle}-1-i\langle\theta_{1},x\rangle\mathbf{1}_{[0,1]}(\|x\|)Q^{1}(dx)\Bigg|\leq J_{1}+J_{2}.
\end{equation*}
For $J_{1}$, we have
\begin{equation*}
J_{1}=\Bigg|\frac{1}{2}\int_{\{0<\|x\|<\delta,\|y\|<\delta\}}\langle\theta_{1},x\rangle^{2}Q_{r}(dx,dy)-\frac{1}{2}\int_{\{0<\|x\|<\delta\}}\langle\theta_{1},x\rangle^{2}Q_{r}(dx,dy)\Bigg|
\end{equation*}
\begin{equation*}
\leq \int_{\{0<\|x\|<\delta\}}\langle\theta_{1},x\rangle^{2}Q_{r}(dx,dy)\leq\|\theta_{1}\|^{2}\int_{\{0<\|x\|<\delta\}}\|x\|^{2}Q^{1}(dx).
\end{equation*}
For $J_{2}$, we have
\begin{equation*}
J_{2}=\Bigg|\int_{\{0<\|x\|<\delta\}}\frac{1}{2}\langle\theta_{1},x\rangle^{2}+e^{i\langle\theta_{1},x\rangle}-1-i\langle\theta_{1},x\rangle\mathbf{1}_{[0,1]}(\|x\|)Q^{1}(dx)\Bigg|,
\end{equation*}
and, by using the multivariate Taylor expansion and noticing that in this case the expansion is only for the variable $x$, we obtain
\begin{equation*}
J_{2}\leq \|\theta_{1}\|^{3}\delta\int_{\{0<\|x\|<\delta\}}\|x\|^{2}Q^{1}(dx).
\end{equation*}
Similar arguments apply to the second addend of the first term of $I_{3}$.
\\Combining all the different results, we get
\begin{equation*}
\lim\limits_{r\rightarrow\infty}\log\Phi_{r}(\theta_{1},\theta_{2})=\log\Phi_{1}(\theta_{1})+\log\Phi_{2}(\theta_{2}), \enspace \enspace \forall \theta_{1},\theta_{2}\in\mathbb{R}^{md},
\end{equation*}
and consequently we obtain the desired result in $(\ref{Mix})$, which concludes the proof.
\subsection*{Proof of Lemma \ref{LEMMA}}
Fix $\epsilon>0$, put $B_{\delta}=\{(x,y)\in\mathbb{R}^{d}\times\mathbb{R}^{d}:\|x\|^{2}+\|y\|^{2}\leq\delta^{2}\}$ and $R_{\delta}=\{(x,y)\in\mathbb{R}^{d}\times\mathbb{R}^{d}:\delta^{2}<\|x\|^{2}+\|y\|^{2}\leq1\}$. Then, we get
\begin{equation*}
\int_{\{0<\|x\|^{2}+\|y\|^{2}\leq 1\}}\|x\|\cdot\|y\|Q_{0t_{n}}(dx,dy)=\int_{B_{\delta}}\|x\|\cdot\|y\|Q_{0t_{n}}(dx,dy)+\int_{R_{\delta}}\|x\|\cdot\|y\|Q_{0t_{n}}(dx,dy):=P_{1}+P_{2}.
\end{equation*}
We will estimate the terms $P_{1}$ and $P_{2}$ separately. Using the stationarity of $Q_{0t_{n}}$ (due to the stationarity of $(X_{t_{n}})_{t_{n}\in\mathbb{R}^{l}}$) and the consistency of the L\'{e}vy measure, we get
\begin{equation*}
|P_{1}|\leq\dfrac{1}{2}\int_{B_{\delta}}\|x\|^{2}Q_{0t_{n}}(dx,dy)+\dfrac{1}{2}\int_{B_{\delta}}\|y\|^{2}Q_{0t_{n}}(dx,dy)
\end{equation*}
\begin{equation*}
\leq \dfrac{1}{2}\int_{\{\|x\|^{2}\leq\delta^{2},y\in\mathbb{R}^{d}\}}\|x\|^{2}Q_{0t_{n}}(dx,dy)+\dfrac{1}{2}\int_{\{\|y\|^{2}\leq\delta^{2},x\in\mathbb{R}^{d}\}}\|y\|^{2}Q_{0t_{n}}(dx,dy)=\int_{\|x\|\leq\delta}\|x\|^{2}Q_{0}(dx).
\end{equation*}
Here $Q_{0}$ is the L\'{e}vy measure of $X_{0}$. Thus, for some appropriately small $\delta_{0}$ we have
\begin{equation}\label{LEMMAEq2}
|P_{1}|=\int_{B_{\delta_{0}}}\|x\|\cdot\|y\|Q_{0t_{n}}(dx,dy)\leq\frac{\epsilon}{2}.
\end{equation}
For the second term, set $c_{0}=\min\left\{\delta_{0},\frac{\epsilon}{8q}\right\}$
 with $q=Q_{0}\left(\|x\|^{2}>\frac{\delta_{0}^{2}}{2}\right)<\infty$. Then, for $C=R_{\delta_{0}}\cap\{\|x\|\cdot\|y\|>c_{0}\}$ we obtain
\begin{equation*}
 |P_{2}|=\int_{C}\|x\|\cdot\|y\|Q_{0t_{n}}(dx,dy)+\int_{R_{\delta_{0}}\setminus C}\|x\|\cdot\|y\|Q_{0t_{n}}(dx,dy)\leq \frac{1}{2}Q_{0t_{n}}(C)+\int_{R_{\delta_{0}}\setminus C}\frac{\epsilon}{8q}Q_{0t_{n}}(dx,dy)
\end{equation*}
 \begin{equation*}
 \leq \frac{1}{2}Q_{0t_{n}}(\|x\|\cdot\|y\|>c_{0})+\frac{\epsilon}{8q}Q_{0t_{n}}(R_{\delta_{0}}\setminus C)
 \end{equation*}
 \begin{equation*}
\leq \frac{1}{2}Q_{0t_{n}}(\|x\|\cdot\|y\|>c_{0})+\frac{\epsilon}{8q}Q_{0t_{n}}\left(\left\{(x,y):\|x\|^{2}>\frac{\delta_{0}^{2}}{2}\right\}\cup\left\{(x,y):\|y\|^{2}>\frac{\delta_{0}^{2}}{2}\right\}\right)
 \end{equation*}
 \begin{equation*}
\leq \frac{1}{2}Q_{0t_{n}}(\|x\|\cdot\|y\|>c_{0})+\frac{\epsilon}{4q}Q_{0t_{n}}\left(\|x\|^{2}>\frac{\delta_{0}^{2}}{2}\right)=\frac{1}{2}Q_{0t_{n}}(\|x\|\cdot\|y\|>c_{0})+\frac{\epsilon}{4}.
 \end{equation*}
For $n$ large enough we have $Q_{0t_{n}}(\|x\|\cdot\|y\|>c_{0})<\frac{\epsilon}{2}$ and therefore
 \begin{equation}\label{LEMMAEq3}
 |P_{2}|=\int_{R_{\delta}}\|x\|\cdot\|y\|Q_{0t_{n}}(dx,dy)<\frac{\epsilon}{2}.
 \end{equation}
 Finally, combining $(\ref{LEMMAEq2})$ and $(\ref{LEMMAEq3})$, and letting $\epsilon\rightarrow0$, we obtain the result of the lemma.
 \subsection*{Proof of Proposition \ref{proMMA}}
 By independence of the random fields $( Y^{k}_{t})_{t\in\mathbb{R}^{l}}$, $k=1,...,r$, we have that the L\'{e}vy-Khintchine representation of $(X_{t})_{t\in\mathbb{R}^{l}}$ can be written as the product of the L\'{e}vy-Khintchine representation of the $(Y^{k}_{t})_{t\in\mathbb{R}^{l}}$, $k=1,...,r$. In other words, for any $n\in\mathbb{N}$ and $\theta_{j}\in\mathbb{R}^{d}$, $j=1,...,n$, we have
 \begin{equation*}
 \mathbb{E}\left[\exp\left(i\sum_{j=1}^{n}\langle \theta_{j}, X_{t_{j}}\rangle\right)\right]=\mathbb{E}\left[\exp\left(i\sum_{j=1}^{n}\langle \theta_{j}, \sum_{k=1}^{r}Y^{k}_{t_{j}}\rangle\right)\right]=\prod_{k=1}^{r}\mathbb{E}\left[\exp\left(i\sum_{j=1}^{n}\langle \theta_{j}, Y^{k}_{t_{j}}\rangle\right)\right].
 \end{equation*}
 Now $(X_{t})_{t\in\mathbb{R}^{l}}$ is stationary since for any $h\in\mathbb{R}^{l}$
 \begin{equation*}
 \mathbb{E}\left[\exp\left(i\sum_{j=1}^{n}\langle \theta_{j}, Y^{k}_{t_{j}}\rangle\right)\right]=\mathbb{E}\left[\exp\left(i\sum_{j=1}^{n}\langle \theta_{j}, Y^{k}_{t_{j}+h}\rangle\right)\right]
 \end{equation*}
 because $( Y^{k}_{t})_{t\in\mathbb{R}^{l}}$ is stationary, for any $k=1,..,r$. Moreover, $(X_{t})_{t\in\mathbb{R}^{l}}$ is ID since it is a sum of independent ID random fields.
 \\ To show that $(X_{t})_{t\in\mathbb{R}^{l}}$ is mixing we proceed as follows. Consider any sequence $(t_{n})_{n\in\mathbb{N}}\in\mathcal{T}$ and consider the joint random field $(X_{0},X_{t_{n}})$. First, notice that the covariance function of the Gaussian part of $(X_{t_{n}})_{t_{n}\in\mathbb{R}^{l}}$, call it $\Sigma_{X}(t_{n})$, is given by the sum of the covariance functions of the $(Y^{k}_{t})_{t\in\mathbb{R}^{l}}$, call them $\Sigma_{Y^{k}}(t_{n})$, $k=1,...,r$. Moreover, notice also that the L\'{e}vy measure of the L\'{e}vy-Khintchine formula of the law $\mathcal{L}(X_{t_{n}},X_{0})$, call it $Q_{X,0t_{n}}$, is given by the sum of the L\'{e}vy measures of the L\'{e}vy-Khintchine formula of the laws $\mathcal{L}(Y^{k}_{t_{n}},Y^{k}_{0})$, call them $Q_{Y^{k},0t_{n}}$, $k=1,...,r$. It is possible to see this in formulae.
 \begin{equation*}
 \mathbb{E}\left[\exp\left(i\langle \theta_{1}, X_{t_{n}}\rangle+i\langle \theta_{2}, X_{0}\rangle\right)\right]=\mathbb{E}\left[\exp\left(i\langle \theta_{1}, \sum_{k=1}^{r}Y^{k}_{t_{n}}\rangle+i\langle \theta_{2}, \sum_{k=1}^{r}Y^{k}_{0}\rangle\right)\right]
 \end{equation*}
 \begin{equation*}
 =\prod_{k=1}^{r}\mathbb{E}\left[\exp\left(i\langle \theta_{1}, Y^{k}_{t_{n}}\rangle+i\langle \theta_{2}, Y^{k}_{0}\rangle\right)\right]=\prod_{k=1}^{r}\exp\bigg\{i\left\langle\binom{\theta_{1}}{\theta_{2}},\binom{\alpha^{k}_{1}}{\alpha^{k}_{2}}\right\rangle-\frac{1}{2}\left\langle\binom{\theta_{1}}{\theta_{2}},\Gamma^{k}(t_{n})\binom{\theta_{1}}{\theta_{2}}\right\rangle
 \end{equation*}
 \begin{equation*}
 +\int_{\mathbb{R}^{2d}}\left(e^{i<\theta_{1},x>+i<\theta_{2},y>}-1-i<\theta_{1},x>\mathbf{1}_{[0,1]}(\|x\|)-i<\theta_{2},y>\mathbf{1}_{[0,1]}(\|y\|)\right)Q_{Y^{k},0t_{n}}(dx,dy)\bigg\}
 \end{equation*}
 \begin{equation*}
 =\exp\bigg\{i\left\langle\binom{\theta_{1}}{\theta_{2}},\binom{\sum_{k=1}^{r}\alpha^{k}_{1}}{\sum_{k=1}^{r}\alpha^{k}_{2}}\right\rangle-\frac{1}{2}\left\langle\binom{\theta_{1}}{\theta_{2}},\sum_{k=1}^{r}\Gamma^{k}(t_{n})\binom{\theta_{1}}{\theta_{2}}\right\rangle
 \end{equation*}
 \begin{equation*}
 +\int_{\mathbb{R}^{2d}}\left(e^{i<\theta_{1},x>+i<\theta_{2},y>}-1-i<\theta_{1},x>\mathbf{1}_{[0,1]}(\|x\|)-i<\theta_{2},y>\mathbf{1}_{[0,1]}(\|y\|)\right)\sum_{k=1}^{r}Q_{Y^{k},0t_{n}}(dx,dy)\bigg\}
 \end{equation*}
 \begin{equation*}
 =\exp\bigg\{i\left\langle\binom{\theta_{1}}{\theta_{2}},\binom{\alpha^{1}_{X}}{\alpha^{2}_{X}}\right\rangle-\frac{1}{2}\left\langle\binom{\theta_{1}}{\theta_{2}},\Gamma_{X}(t_{n})\binom{\theta_{1}}{\theta_{2}}\right\rangle
 \end{equation*}
 \begin{equation*}
 +\int_{\mathbb{R}^{2d}}\left(e^{i<\theta_{1},x>+i<\theta_{2},y>}-1-i<\theta_{1},x>\mathbf{1}_{[0,1]}(\|x\|)-i<\theta_{2},y>\mathbf{1}_{[0,1]}(\|y\|)\right)Q_{X,0t_{n}}(dx,dy)\bigg\}
 \end{equation*}
 where
 \begin{equation*}
 \Gamma^{k}(t_{n}) = 
  \begin{pmatrix}
   \Sigma(0) & \Sigma_{Y^{k}}(t_{n})  \\
   \Sigma_{Y^{k}}(t_{n}) & \Sigma(0)
  \end{pmatrix}.
  \end{equation*}
  In order to prove that $(X_{t})_{t\in\mathbb{R}^{l}}$ is mixing we need to show that conditions $(MM1)$ and $(MM2')$ hold. However, these conditions hold for each $\Sigma_{Y^{k}}(t_{n})$ and $Q_{Y^{k},0t_{n}}$, where $k=1,...,r$. Moreover, since both the covariance matrix function $\Sigma_{X}(t_{n})$ of the Gaussian part of $(X_{t})_{t\in\mathbb{R}^{l}}$ and its L\'{e}vy measure $Q_{X,0t_{n}}$ are sums of $\Sigma_{Y^{k}}(t_{n})$ and $Q_{Y^{k},0t_{n}}$ respectively, for $k=1,...,r$, then we have that these conditions hold also for them. Hence, $(X_{t})_{t\in\mathbb{R}^{l}}$ is mixing.
\section*{Appendix B: Proofs of Section 3}
\subsection*{Proof of Lemma \ref{lemmaMMA}}
The argument of the proof of Lemma 3.4 of \cite{FuSt} can be extended to our setting. To this end, notice that we can write
\begin{equation*}
\begin{pmatrix}
  X_{0} \\
  X_{t_{n}}
 \end{pmatrix}=\int_{S}\int_{\mathbb{R}^{l}}\begin{pmatrix}
   f(A,-s) \\
   f(A,t_{n}-s)
  \end{pmatrix}\Lambda(dA,ds), \qquad t_{n}\in\mathbb{R}^{l}.
\end{equation*}
From this and from Theorem \ref{TheoremRajRos} it is possible to compute the covariance matrix function of the Gaussian part of $(X_{t})_{t\in\mathbb{R}^{l}}$ by
\begin{equation*}
\Gamma(t_{n}) = 
 \begin{pmatrix}
  \Sigma(0) & \Sigma(t_{n})  \\
  \Sigma(t_{n}) & \Sigma(0)
 \end{pmatrix},
 \end{equation*}
 where
\begin{equation*}
\Sigma(t_{n})=\int_{S}\int_{\mathbb{R}^{l}}f(A,-s)\Sigma f(A,t_{n}-s)'ds\pi(dA), \qquad t_{n}\in\mathbb{R}^{l}.
\end{equation*}
The L\'{e}vy measure $Q_{0t_{n}}$ of $\mathcal{L}(X_{0},X_{t_{n}})$ is given by 
\begin{equation*}
Q_{0t_{n}}=\int_{S}\int_{\mathbb{R}^{l}}\int_{\mathbb{R}^{d}}\normalfont\textbf{1}_{B}(f(A,-s)x,f(A,t_{n}-s)x)Q(dx)ds\pi(dA),
\end{equation*}
for all Borel sets $B\subseteq\mathbb{R}^{2q}\setminus\{0\}$, using again Theorem \ref{TheoremRajRos}. Therefore, given this explicit representation of $Q_{0t_{n}}$ we have that
\begin{equation*}
\int_{\mathbb{R}^{2q}}\min (1,\|x\|\cdot\|y\|)Q_{0t_{n}}(dx,dy)=\int_{S}\int_{\mathbb{R}^{l}}\int_{\mathbb{R}^{d}}\min(1,\|f(A,-s)x\|\cdot\|f(A,t_{n}-s)x\|)Q(dx)ds\pi(dA).
\end{equation*}
Now by using Corollary \ref{corol} we complete the proof.
\subsection*{Proof of Theorem \ref{MMA}}
We generalise the arguments of given in the proof of Theorem 3.5 of \cite{FuSt}.
Following Lemma \ref{lemmaMMA}, in order to prove this theorem it is sufficient to show that
\begin{equation*}
\|\Sigma(t_{n})\|=\bigg\| \int_{S}\int_{\mathbb{R}^{l}}f(A,-s)\Sigma f(A,t_{n}-s)'ds\pi(dA) \bigg\|\stackrel{n\rightarrow\infty}{\rightarrow}0
\end{equation*}
and
\begin{equation*}
\int_{\mathbb{R}^{2d}}\min (1,\|x\|\cdot\|y\|)Q_{0t_{n}}(dx,dy)
\end{equation*}
\begin{equation*}
=\int_{S}\int_{\mathbb{R}^{l}}\int_{\mathbb{R}^{d}}\min(1,\|f(A,-s)x\|\cdot\|f(A,t_{n}-s)x\|)Q(dx)ds\pi(dA)\stackrel{n\rightarrow\infty}{\rightarrow}0,
\end{equation*}
for any $(t_{n})_{n\in\mathbb{N}}\in\mathcal{T}$.\\
First, we concentrate on proving that $\|\Sigma(t_{n})\|\stackrel{n\rightarrow\infty}{\rightarrow}0.$ Consider that by the existence of the MMA field we have (see Theorem \ref{TheoremRajRos})
\begin{equation*}
\int_{S}\int_{\mathbb{R}^{l}}\Big\|f(A,t-s)\Sigma^{\frac{1}{2}}\Big\|^{2}ds\pi(dA)<\infty,
\end{equation*}
for any $t\in\mathbb{R}^{l}$, where $\Sigma^{\frac{1}{2}}$ denotes the unique square root of $\Sigma$. Therefore for any $t\in\mathbb{R}^{l}$, the function $g_{t}:S\times\mathbb{R}^{l}\rightarrow\mathbb{R}, (A,s)\mapsto\|f(A,t-s)\Sigma^{\frac{1}{2}}\|$ is an element of $L^{2}(S\times\mathbb{R}^{l},\mathcal{B}(S\times\mathbb{R}^{l}),\pi\otimes\lambda^{l};\mathbb{R})$. The fact that the measure $\pi\otimes\lambda^{l}$ is $\sigma$-finite implies that every $L^{2}$-function can be approximated (in the $L^{2}$-norm) by an elementary function in
\begin{equation*}
\mathcal{E}:=\bigg\{f\in L^{2}(S\times\mathbb{R}^{l},\mathcal{B}(S\times\mathbb{R}^{l}),\pi\otimes\lambda^{l};\mathbb{R}):f=\sum_{i=1}^{k}c_{i}\normalfont\textbf{1}_{D_{i}\times R_{i}}, k\in\mathbb{N},
\end{equation*}
\begin{equation*}
c_{i}\in\mathbb{R},D_{i}\in\mathcal{B}(S), R_{i}\in\mathcal{B}(\mathbb{R}^{l}), i=1,...,k\bigg\}.
\end{equation*}
Let us now fix an arbitrary $\epsilon>0$ and choose an elementary function $\tilde{g}\in\mathcal{E}$ such that
\begin{equation*}
\|g_{0}-\tilde{g}\|_{L^{2}}=\left(\int_{S}\int_{\mathbb{R}^{l}}|g_{0}(A,s)-\tilde{g}(A,s)|^{2}ds\pi(dA)\right)^{\frac{1}{2}}<\epsilon.
\end{equation*}
Now we have that for any $t_{n}\in\mathbb{R}^{l}$
\begin{equation*}
\|\Sigma(t_{n})\|=\bigg\| \int_{S}\int_{\mathbb{R}^{l}}f(A,-s)\Sigma^{\frac{1}{2}} \left(f(A,t_{n}-s)\Sigma^{\frac{1}{2}}\right)'ds\pi(dA) \bigg\|
\end{equation*}
\begin{equation*}
\leq \int_{S}\int_{\mathbb{R}^{l}}g_{0}(A,s)\cdot g_{t}(A,s)ds\pi(dA)
\end{equation*}
\begin{equation*}
= \int_{S}\int_{\mathbb{R}^{l}}(g_{0}(A,s)-\tilde{g}(A,s))\cdot g_{t}(A,s)+\tilde{g}(A,s)(g_{t}(A,s)-\tilde{g}(A,s-t_{n}))+\tilde{g}(A,s)\tilde{g}(A,s-t_{n})ds\pi(dA).
\end{equation*}
By the Cauchy-Schwarz inequality we obtain
\begin{equation*}
\|\Sigma(t_{n})\|\leq\epsilon\cdot\|g_{t_{n}}\|_{L^{2}}+\|\tilde{g}\|_{L^{2}}\cdot\left(\int_{S}\int_{\mathbb{R}^{l}}|g_{t_{n}}(A,s)-\tilde{g}(A,s-t_{n})|^{2}ds\pi(dA)\right)^{\frac{1}{2}}+\int_{S}\int_{\mathbb{R}^{l}}|\tilde{g}(A,s)\tilde{g}(A,s-t_{n})|^{2}ds\pi(dA).
\end{equation*}
Now notice that
\begin{equation*}
\left(\int_{S}\int_{\mathbb{R}^{l}}|g_{t_{n}}(A,s)-\tilde{g}(A,s-t_{n})|^{2}ds\pi(dA)\right)^{\frac{1}{2}}=\left(\int_{S}\int_{\mathbb{R}^{l}}|g_{0}(A,s-t_{n})-\tilde{g}(A,s-t_{n})|^{2}ds\pi(dA)\right)^{\frac{1}{2}}<\epsilon
\end{equation*}
and that $\|g_{t_{n}}\|_{L^{2}}=\|g_{0}\|_{L^{2}}$, by a simple change of variables. In addition, $\|\tilde{g}\|_{L^{2}}\leq\|g_{0}\|_{L^{2}}+\|\tilde{g}-g_{0}\|_{L^{2}}<\|g_{0}\|_{L^{2}}+\epsilon$.
\\ Finally, it is possible to see that
\begin{equation*}
\int_{S}\int_{\mathbb{R}^{l}}|\tilde{g}(A,s)\tilde{g}(A,s-t_{n})|^{2}ds\pi(dA)=0,
\end{equation*}
for sufficiently large $t_{n}$ (or equivalently for sufficiently large $n$, where $t_{n}$ is an element of $(t_{n})_{n\in\mathbb{N}}\in\mathcal{T}$). This is because $\tilde{g}\in\mathcal{E}$. In particular, given $\tilde{g}\in\mathcal{E}$ then $\tilde{g}(A,s)=\sum_{i=1}^{k}c_{i}\normalfont\textbf{1}_{D_{i}\times R_{i}}(A,s), k\in\mathbb{N}$, hence $\tilde{g}(A,s-t_{n})=0$ for $t_{n}$ sufficiently large using the fact that the rectangles $R_{i}$ cannot cover the whole $\mathbb{R}^{l}$ since $\int_{S}\int_{\mathbb{R}^{l}}\|\tilde{g}\|^{2}ds\pi(dA)<\infty$, for any $i=1,...,k$. Therefore, we have 
\begin{equation*}
\|\Sigma(t_{n})\|<\epsilon\cdot\|g_{0}\|_{L^{2}}+(\epsilon+\|g_{0}\|_{L^{2}})\cdot\epsilon,
\end{equation*}
for sufficiently large $n$. This yields $\|\Sigma(t_{n})\|\rightarrow0$ as $n\rightarrow\infty$, for any $(t_{n})_{n\in\mathbb{N}}\in\mathcal{T}$.\\
We now move to the second objective of the proof. Indeed, we now prove that
\begin{equation*}
\int_{\mathbb{R}^{2q}}\min (1,\|x\|\cdot\|y\|)Q_{0t_{n}}(dx,dy)\stackrel{n\rightarrow\infty}{\rightarrow}0,
\end{equation*}
for any $(t_{n})_{n\in\mathbb{N}}\in\mathcal{T}$.\\
Consider an arbitrary $\epsilon>0$ and set $B_{r}:=\{(x,y)\in\mathbb{R}^{q}\times\mathbb{R}^{q}:\|x\|^{2}+\|y\|^{2}\leq r^{2}\}$. Recall now the argument used to prove $(\ref{TRUNCATION})$. In that argument we did not assume that the random field was mixing, but only that was stationary and ID. Thus, we have that for any $(t_{n})_{n\in\mathbb{N}}\in\mathcal{T}$ the following holds
\begin{equation*}
\sup\limits_{n\geq n_{0}}Q_{0t_{n}}(\mathbb{R}^{2q}\setminus B_{R})\leq \epsilon,
\end{equation*}
for some $R>1$ and some $n_{0}>0$. Therefore, for all $n\geq n_{0}$, we obtain that
\begin{equation*}
\int_{\mathbb{R}^{2q}}\min (1,\|x\|\cdot\|y\|)Q_{0t_{n}}(dx,dy)\leq  \int_{B_{R}}\min (1,\|x\|\cdot\|y\|)Q_{0t_{n}}(dx,dy)+\epsilon.
\end{equation*}
Now notice that when $\max\{\|u\|,\|v\|\}\leq R$ we have
\begin{equation*}
\min\{\|u\|\cdot\|v\|,1\}\leq R\cdot\min\{\|u\|,1\}\cdot\min\{\|v\|,1\}.
\end{equation*}
Hence, we have
\begin{equation*}
\int_{B_{R}}\min (1,\|x\|\cdot\|y\|)Q_{0t_{n}}(dx,dy)\leq R\cdot \int_{B_{R}}\min (1,\|x\|)\cdot\min (1,\|y\|)Q_{0t_{n}}(dx,dy)
\end{equation*}
\begin{equation*}
\leq R\cdot \int_{S}\int_{\mathbb{R}^{l}}\int_{\mathbb{R}^{d}}\min(1,\|f(A,-s)x\|)\cdot\min(1,\|f(A,t_{n}-s)x\|)Q(dx)ds\pi(dA).
\end{equation*}
By the existence of the MMA random field, we have that for any $t_{n}\in\mathbb{R}^{l}$ the function $h_{t_{n}}:S\times\mathbb{R}^{l}\times\mathbb{R}^{d}\rightarrow\mathbb{R}$, $h_{t_{n}}(A,s,x):=\min(1,\|f(A,t_{n}-s)x\|)$ is an element of $L^{2}(S\times\mathbb{R}^{l}\times\mathbb{R}^{d},\mathcal{B}(S\times\mathbb{R}^{l}\times\mathbb{R}^{d}),\pi\otimes\lambda^{l}\otimes Q;\mathbb{R})$. Also, the fact that every L\'{e}vy measure is $\sigma$-finite implies that the product measure $\pi\otimes\lambda^{l}\otimes Q$ is $\sigma$-finite as well and therefore it is possible to use the same approximation argument used above in the first part of this proof to show that
\begin{equation*}
\int_{S}\int_{\mathbb{R}^{l}}\int_{\mathbb{R}^{d}}\min(1,\|f(A,-s)x\|)\cdot\min(1,\|f(A,t_{n}-s)x\|)Q(dx)ds\pi(dA)\stackrel{n\rightarrow\infty}{\rightarrow}0,
\end{equation*}
for any $(t_{n})_{n\in\mathbb{N}}\in\mathcal{T}$, which completes the proof.
\section*{Appendix C: Proofs of Section 4}
\subsection*{Proof of Theorem \ref{Theorem-ergodicity}}
This proof is a multivariate and a random field extension of the proof of Theorem 1 of \cite{1997RoZa}.\\
``$\Rightarrow$": It is well known that any weakly mixing random field is ergodic.
\\ ``$\Leftarrow$": For the other direction we argue as follows. Let $(X_{t})_{t\in\mathbb{R}^{l}}$ be an ergodic $\mathbb{R}^{d}$-valued stationary ID random field. In this proof we will work with $j,k=1,...,d$ and we will not repeat it every time. We showed before that 
\begin{equation*}
\tau(t)=\left(\tau^{(jk)}(t)\right)_{j,k=1,...,d}
\end{equation*}
with
\begin{equation}\label{tau}
\tau^{(jk)}(t):=\tau\left(X_{0}^{(k)},X_{t}^{(j)}\right)=\sigma_{t}^{jk}+\int_{\mathbb{R}^{2}}(e^{ix}-1)\overline{(e^{iy}-1)}Q_{0t}^{jk}(dx,dy)
\end{equation}
is the autocodifference field of $(X_{t})_{t\in\mathbb{R}^{l}}$. Since $(X_{t})_{t\in\mathbb{R}^{l}}$ is an $\mathbb{R}^{d}$-valued ID and stationary random field, we have that
\begin{equation*}
\tau\left(X_{s}^{(k)},X_{t}^{(j)}\right)=\tau\left(X_{0}^{(k)},X_{t-s}^{(j)}\right).
\end{equation*}
Hence, as shown before in Proposition \ref{pro-ergodicity}, the function
\begin{equation*}
\tau^{(jk)}(t)=\tau\left(X_{0}^{(k)},X_{t}^{(j)}\right)
\end{equation*}
is nonnegative definite and $\tau^{(jk)}(0)=-\log\left(\mathbb{E}[e^{iX_{0}^{(j)}}] \cdot \mathbb{E}[e^{iX_{0}^{(k)}}]  \right)$
which is a constant. Hence, we can use Bochner's theorem, which implies that there exists a finite Borel measure $v$ on $\mathbb{R}^{l}$ such that
\begin{equation*}
\tau^{(jk)}(t)=\int_{\mathbb{R}^{l}}e^{i\langle t,\lambda\rangle}v(d\lambda).
\end{equation*}
Thus,
\begin{equation*}
\mathbb{E}[e^{i(X^{(j)}_{0}-X^{(k)}_{t})}]\left(\mathbb{E}[e^{iX_{0}^{(j)}}] \cdot \mathbb{E}[e^{iX_{0}^{(k)}}] \right)^{-1}=e^{\tau^{(jk)}(t)}=e^{\hat{v}(t)}=\widehat{\exp(v)}(t),
\end{equation*}
where $\exp(v)=\sum_{n=o}^{\infty}(v^{*n}/n!),v^{*0}=\delta_{0}$ and the symbol `` $\hat{}$ " denotes the Fourier transform. Notice that the last equality comes from the convolution theorem. Hence, $\tau^{(jk)}=\hat{v}$. Moreover, since both terms on the RHS of $(\ref{tau})$ are non-negative definite thanks to Proposition \ref{pro-ergodicity}, again by Bochner's theorem there exists finite Borel measures $v_{G}$ and $v_{P}$ on $\mathbb{R}^{l}$ such that
\begin{equation}\label{tau=v}
\tau^{(jk)}(t)=\hat{v}_{G}(t)+\hat{v}_{P}(t).
\end{equation}
Thus $v=v_{G}+v_{P}$. The ergodicity of the process implies that
\begin{equation*}
\frac{1}{(2T)^{l}}\int_{(-T,T]^{l}}\mathbb{E}\left[e^{i(X_{0}^{(j)}-X_{t}^{(k)})}\right]dt\stackrel{T\rightarrow\infty}{\rightarrow}\mathbb{E}\left[e^{iX_{0}^{(j)}}\right]\mathbb{E}\left[e^{-iX_{0}^{(k)}}\right].
\end{equation*}
Hence, combining this result with Lemma \ref{lemma-mu(0)}, we obtain
\begin{equation*}
\frac{1}{(2T)^{l}}\int_{(-T,T]^{l}}\widehat{\exp(v)}(t)dt\stackrel{T\rightarrow\infty}{\longrightarrow}\exp(v_{G}+v_{P})(\{0\})=1,
\end{equation*}
which implies that
\begin{equation}\label{eq3.4}
v_{G}(\{0\})+v_{P}(\{0\})=0 \quad\textnormal{and}\quad (v_{G}+v_{P})*(v_{G}+v_{P})(\{0\})=0,
\end{equation}
and so $v_{G}*v_{G}(\{0\})=0$. Therefore, we have that
\begin{equation*}
\frac{1}{(2T)^{l}}\int_{(-T,T]^{l}}(\sigma_{t}^{jk})^{2}dt=\frac{1}{(2T)^{l}}\int_{(-T,T]^{l}}(\hat{v}_{G}(t))^{2}dt\stackrel{T\rightarrow\infty}{\longrightarrow}v_{G}*v_{G}(\{0\})=0.
\end{equation*}
Since $\sigma_{t}^{jk}$ is real, we deduce from Lemma \ref{lemma-weakmixing} that there exists a set $D$ of density one in $\mathbb{R}^{l}$, such that
\begin{equation}\label{eq3.5}
\lim\limits_{n\rightarrow\infty}\sigma^{jk}_{t_{n}}=0,\qquad\textnormal{for any}\quad (t_{n})_{n\in\mathbb{N}}\in\mathcal{T}_{D}.
\end{equation}
Now we would like to prove a similar result for the L\'{e}vy measure of $\mathcal{L}(X_{0}^{(j)},X_{t_{n}}^{(k)})$, so that we can then apply Corollary \ref{coroll}, which will give us weak mixing of our random field $(X_{t})_{t\in\mathbb{R}^{l}}$. By equations (\ref{tau}), (\ref{tau=v}), (\ref{eq3.4}) and Lemma \ref{lemma-mu(0)} we have that
\begin{equation*}
\frac{1}{(2T)^{l}}\int_{(-T,T]^{l}}\int_{\mathbb{R}^{2}}(e^{ix}-1)\overline{(e^{iy}-1)}Q_{0t}^{jk}(dx,dy)dt=\frac{1}{(2T)^{l}}\int_{(-T,T]^{l}}\hat{v}_{P}(t)dt\stackrel{T\rightarrow\infty}{\longrightarrow}v_{P}(\{0\})=0.
\end{equation*}
When taking the real part of the first term above we get that
\begin{equation}\label{eq3.6}
\frac{1}{(2T)^{l}}\int_{(-T,T]^{l}}\int_{\mathbb{R}^{2}}[(\cos x-1)(\cos y -1)+\sin x\sin y]Q_{0t}^{jk}(dx,dy)dt\stackrel{T\rightarrow\infty}{\longrightarrow}0.
\end{equation}
Now consider the two integrands of eq. ($\ref{eq3.6}$). By stationarity and Proposition \ref{pro-ergodicity}, the functions
\begin{equation*}
t\rightarrow\int_{\mathbb{R}^{2}}(\cos x-1)(\cos y -1)Q_{0t}^{jk}(dx,dy),
\quad\textnormal{and}\quad
t\rightarrow\int_{\mathbb{R}^{2}}\sin x\sin yQ_{0t}^{jk}(dx,dy)
\end{equation*}
are non-negative definite. By Bochner's theorem, this implies that there exist finite Borel measures $\lambda_{1}$ and $\lambda_{2}$ such that equation (\ref{eq3.6}) can be written as
\begin{equation*}
\frac{1}{(2T)^{l}}\int_{(-T,T]^{l}}\hat{\lambda}_{1}dt+\frac{1}{(2T)^{l}}\int_{(-T,T]^{l}}\hat{\lambda}_{2}dt\stackrel{T\rightarrow\infty}{\longrightarrow}0,
\end{equation*}
and using Lemma \ref{lemma-mu(0)} we obtain $\lambda_{1}(\{0\})=\lambda_{2}(\{0\})=0$. Focusing on the first one, we conclude that
\begin{equation}\label{eq3.7}
\frac{1}{(2T)^{l}}\int_{(-T,T]^{l}}\int_{\mathbb{R}^{2}}(\cos x-1)(\cos y -1)Q_{0t}^{jk}(dx,dy)dt\stackrel{T\rightarrow\infty}{\longrightarrow}\lambda_{1}(\{0\})=0.
\end{equation}
Define $R_{T}(dx,dy)=\frac{1}{(2T)^{l}}\int_{(-T,T]^{l}}Q_{0t}^{jk}(dx,dy)dt$. Then we have
\begin{equation}\label{eq3.8}
\int_{\mathbb{R}^{2}}(\cos x-1)(\cos y -1)R_{T}(dx,dy)\stackrel{T\rightarrow\infty}{\longrightarrow}0.
\end{equation}
It is possible to notice that the family of finite measures $(R_{T}|_{K^{c}_{\delta}})_{T>0}$ is weakly relatively compact for every $\delta>0$. This is because by Lemma \ref{lemma-compact} $(Q_{0t}^{jk}|_{K^{c}_{\delta}})_{t\in\mathbb{R}^{l}}$ is weakly relatively compact for every $\delta>0$. The goal is now to show that
\begin{equation}\label{eq3.9}
\int_{\mathbb{R}^{2}}\min(|xy|,1)R_{T}(dx,dy)\stackrel{T\rightarrow\infty}{\longrightarrow}0.
\end{equation}
So, let $T_{n}\rightarrow\infty$, $T_{n}\in\mathbb{R}$. Using the diagonalization procedure we can find a subsequence $(T'_{n})$ of $(T_{n})$ and a measure $R$ on $\mathbb{R}^{2}\setminus\{0\}$ such that $R_{T'_{n}}|_{K_{\delta}^{c}}\Rightarrow R|_{K_{\delta}^{c}}$ as $n\rightarrow\infty$ for every $\delta>0$. Now, notice that $(\cos x-1)(\cos y -1)\geq 0$ and $(\cos x-1)(\cos y -1)= 0$ if $x=2\pi k$ or $y=2\pi k$, for $k\in\mathbb{N}$. Moreover, by eq. (\ref{eq3.8}) we have that
\begin{equation*}
\int_{K_{\delta}^{c}}(\cos x-1)(\cos y -1)R(dx,dy)=0,
\end{equation*}
for every $\delta>0$. Therefore, the measure $R$ is concentrated on the set of lines $\{(x,y):x\in2\pi\mathbb{Z}\,\,\textnormal{ or }\,\,y\in2\pi\mathbb{Z} \}$. Since the random field is stationary, the projections of $Q_{0t}$ (and of $Q_{0t}^{jk}$) onto the first and second axis (excluding zero) are equal to $Q_{0}$ (and to $Q_{0}^{j}$ for the first axis and to $Q_{0}^{k}$ for the second). The same holds for $R_{T}$ and for $R$.
\\ Suppose for now that $Q_{0}(\{x=(x_{1},...,x_{d})'\in\mathbb{R}^{d}:\exists j\in\{1,...,d\},x_{j}\in2\pi\mathbb{Z}\})=0$, which implies that for $\alpha\in\mathbb{Z}\setminus\{0\}$ we have $Q_{0}^{jk}(2\pi \alpha\times\mathbb{R})=Q_{0}^{jk}(\mathbb{R}\times2\pi \alpha)=0$. Then $R$ must be concentrated on the axes of $\mathbb{R}^{2}$. Hence, for every $\delta>0$ such that $R(\{(x,y):x^{2}+y^{2}=\delta^{2} \})=0$,
\begin{equation*}
\limsup\limits_{n\rightarrow\infty} \int_{\mathbb{R}^{2}}\min(|xy|,1)R_{T'_{n}}(dx,dy)
\end{equation*}
\begin{equation*}
\leq \limsup\limits_{n\rightarrow\infty} \int_{K_{\delta}^{c}}\min(|xy|,1)R_{T'_{n}}(dx,dy)+\sup\limits_{T>0}\frac{1}{(2T)^{l}}\int_{(-T,T]^{l}}\int_{\mathbb{R}^{2}}|xy|Q_{0t}^{jk}(dx,dy)dt.
\end{equation*}
Eq. (\ref{eq1.4}) implies that the last quantity can be made arbitrarily small. Hence, (\ref{eq3.9}) follows. From eq. (\ref{eq3.9}) and Lemma \ref{lemma-weakmixing} we obtain that there exists a set $D'$ of density one in $\mathbb{R}^{l}$ such that
\begin{equation}\label{eq3.10}
\lim\limits_{n\rightarrow\infty}\int_{\mathbb{R}^{2}}\min(|xy|,1)Q^{jk}_{0t_{n}}(dx,dy)=0,\qquad\textnormal{for any}\quad (t_{n})_{n\in\mathbb{N}}\in\mathcal{T}_{D'}.
\end{equation}
In case $D$ of eq. (\ref{eq3.5}) and $D'$ of eq. (\ref{eq3.10}) are different this is not a problem because the intersection of two (or a finite number) of density one sets is again a density one set. \\Hence, following the proof of Corollary \ref{corol} (and using its weak mixing version, namely Corollary \ref{coroll}) we obtain that $(X_{t})_{t\in\mathbb{R}^{l}}$ is weakly mixing with the additional assumption that $Q_{0}(\{x=(x_{1},...,x_{d})'\in\mathbb{R}^{d}:\exists j\in\{1,...,d\},x_{j}\in2\pi\mathbb{Z}\})=0$. However, this assumption can be eliminated by first using the fact that if $(X_{t})_{t\in\mathbb{R}^{l}}$ is ergodic then $(M_{a}X_{t})_{t\in\mathbb{R}^{l}}$ is ergodic too, and then by following the arguments of Theorem \ref{atom}. In particular, let $Z=\{z=(z_{1},...,z_{j})\in\mathbb{R}^{d}: z_{j}=2\pi k/y_{j}\enspace\forall j\in\{1,..,d\},\textnormal{ where }k\in\mathbb{Z}\textnormal{ and $y=(y_{1},...,y_{j})$ is an atom of $Q_{0}$}\}$. The set $Z$ is countable and hence there exists a nonzero $a\in\mathbb{R}^{d}\setminus Z$. Consider the random field $(M_{a}X_{t})_{t\in\mathbb{R}^{l}}$ and let $Q^{a}_{0}$ be the L\'{e}vy measure of $M_{a}X_{0}$. Then $Q^{a}_{0}$ has no atoms in the set $(\{x=(x_{1},...,x_{d})'\in\mathbb{R}^{d}:\exists j\in\{1,...,d\},x_{j}\in2\pi\mathbb{Z}\})$; since $(M_{a}X_{t})_{t\in\mathbb{R}^{l}}$ is also ergodic, it is weakly mixing by the arguments of this proof. Therefore, $(X_{t})_{t\in\mathbb{R}^{l}}$ is weakly mixing and the proof is complete.
\section*{Acknowledgement}
RP would like to thank the CDT in MPE and the Grantham Institute for providing funding for this research.

\end{document}